\documentclass[11pt]{amsart}
\usepackage{amsmath}
\usepackage{amsfonts}
\usepackage{amssymb}
\usepackage{graphicx}
\usepackage{float, verbatim}
\usepackage{enumerate}
\usepackage{color}

\renewcommand{\geq}{\geqslant}
\renewcommand{\leq}{\leqslant}
\renewcommand{\epsilon}{\varepsilon}

\allowdisplaybreaks
\newtheorem{thm}{Theorem}[section]
\newtheorem{lma}[thm]{Lemma}
\newtheorem{cor}[thm]{Corollary}

\newtheorem{prop}[thm]{Proposition}

\newtheorem{rem}[thm]{Remark}
\newtheorem{ques}[thm]{Question}

\newtheorem{example}[thm]{Example}

\numberwithin{equation}{section}

\begin{document}

\title[New dimension spectra]{New dimension spectra:\\ finer information on scaling and homogeneity}


\author[J. M. Fraser]{Jonathan M. Fraser}
\address{Mathematical Institute\\ The University of St Andrews\\ St Andrews\\ KY16 9SS\\ Scotland }
\urladdr{http://www.mcs.st-and.ac.uk/~jmf32/}
\curraddr{}
\email{jmf32@st-andrews.ac.uk}
\thanks{The work of JMF was partially supported by the \emph{Leverhulme Trust Research Fellowship} (RF-2016-500).  This work began  while both authors were at the University of Manchester and they are grateful for the inspiring atmosphere they enjoyed during their time there.  They are also grateful to Chris Miller for posing interesting questions.}

\author[H. Yu]{Han Yu}
\address{Mathematical Institute\\ The University of St Andrews\\ St Andrews\\ KY16 9SS\\ Scotland }
\curraddr{}
\email{hy25@st-andrews.ac.uk}
\thanks{}

\date{}



\begin{abstract}We introduce a new dimension spectrum motivated by the Assouad dimension; a familiar notion of dimension which, for a given metric space, returns the minimal exponent $\alpha\geq 0$ such that for any pair of scales $0<r<R$, any ball of radius $R$ may be covered by a constant times $(R/r)^\alpha$ balls of radius $r$.  To each $\theta \in (0,1)$, we associate the appropriate analogue of the Assouad dimension with the restriction that the two scales $r$ and $R$ used in the definition satisfy $\log R/\log r = \theta$.  The resulting `dimension spectrum' (as a function of $\theta$) thus gives finer geometric information regarding the scaling structure of the space and, in some precise sense, interpolates between the upper box dimension and the Assouad dimension.  This latter point is particularly useful because the spectrum is  generally better behaved than the Assouad dimension.  We also consider the corresponding `lower spectrum', motivated by the lower dimension, which acts as a dual to the Assouad spectrum. 

We conduct a detailed study of these dimension spectra; including analytic, geometric, and measureability properties.  We also compute the spectra explicitly for some common examples of fractals including decreasing sequences  with decreasing gaps and  spirals with sub-exponential and monotonic winding.  We also give several  applications of our results, including: dimension distortion estimates under bi-H\"older maps for Assouad dimension and the provision of  new bi-Lipschitz invariants. 
\end{abstract}

\maketitle

\tableofcontents

\section{New dimension spectra, summary of results, and organisation of paper}

The Assouad dimension is a fundamental notion of dimension used to study fractal objects in a wide variety of contexts.  It was popularised by Assouad in the 1970s \cite{assouadphd, assouad} and subsequently took on significant importance in embedding theory.  Recall the famous \emph{Assouad Embedding Theorem} which states that if $(X,d)$ is a metric space with the doubling property (equivalently, with finite Assouad dimension), then $(X, d^\varepsilon)$ admits a bi-Lipschitz embedding into some finite dimensional Euclidean space for any $\varepsilon \in (0,1)$.  The notion we now call Assouad dimension does go back further, however, to Larman's work in the 1960s \cite{Larman1, Larman2} and even to Bouligand's 1928 paper \cite{Bouligand}.  It is also worth noting that, due to its deep connections with tangents (see \cite{mackaytyson}), it is intimately related to pioneering work of Furstenberg on micro-sets which goes back to the 1960s, see \cite{Furstenberg08}. Roughly speaking, the Assouad dimension assigns a number to a given metric space which quantifies the most difficult location and scale at which to cover the space.  More precisely, it considers two scales $0<r<R$ and finds the maximal exponential growth rate of $N(B(x,R), r)$ as $R$ and $r$ decrease, where $N(E,r)$ is the minimal number of $r$-balls required to cover a set $E$. 

The Assouad dimension has found important applications in a wide variety of contexts, including a sustained importance in embedding theory, see \cite{Robinson, Olson2002, Olson2010}.  It is also central to quasi-conformal geometry, see \cite{Heinonen01, Tyson, mackaytyson}, and has recently been gaining significant attention in the literature on fractal geometry and geometric measure theory, see for example  \cite{Mackay, Luukkainen, Fraser, FraserOrponen, kaenmakiassouad, kaenmakiassouad2, rajala}.  However, its application and interest does not end there. For example, in the study of fractional Hardy inequalities,  if the boundary of a domain in $\mathbb{R}^d$ has Assouad dimension less than or equal to $d-p$, then the domain admits the fractional $p$-Hardy inequality, see \cite{Aikawa91, Koskela03, Lehrback13}.  Also, Hieronymi and  Miller have recently used the Assouad dimension to study nondefinability problems relating to expansions in the real number field, see \cite{miller}.  There are also connections between the Assouad dimension and problems in arithmetic combinatorics, for example the existence of arithmetic progressions or asymptotic arithmetic progressions, see \cite{patches, kota}.

Since it is an \emph{extremal} quantity, the Assouad dimension gives rather coarse information about the space and is often very large; larger than the other familiar notions of dimension such as Hausdorff and box-counting dimension.  Also, despite the fact that two scales are used ($r$ and $R$), the Assouad dimension returns no information about \emph{which} scales `see' the maximal exponential growth rate described above.  In this paper we propose a programme to tackle these  problems: we fix the relationship between the scales $R$ and $r$ and then compute the corresponding \emph{restricted} Assouad dimension by  only considering pairs of scales with this fixed relationship.  More precisely, for a fixed $\theta \in (0,1)$, we look for the maximal exponential growth rate of $N(B(x,R), r)$ as $R$ decreases and $r$ is defined by $\log R/\log r = \theta$.  One can then vary $\theta$ and obtain a spectrum of dimensions for the given metric space which can be viewed as providing finer geometric information about the (lack of) homogeneity present and a more complete picture of how the space scales.  One may  also be able to pick out $\theta$s which `see' the Assouad dimension, i.e., values where the spectrum reaches the true Assouad dimension.  If the Assouad dimension is `seen' by the spectrum, then we are able to glean more information about the Assouad dimension because the spectrum is generally better behaved than the Assouad dimension, see for example Theorem \ref{Holderassouadcor}.

Another key motivation for  this work is that the finer the information we are able glean concerning the scaling structure of the space, the better the applications should be.  In particular, we believe that the notions we introduce and study here should bear fruit in other areas where the Assouad dimension already plays a role; such as embedding theory, quasi-conformal geometry, and geometric measure theory.

We begin by considering how these spectra behave as functions of $\theta$ for arbitrary sets.  Some of the notable results we obtain in this direction include:
\begin{enumerate}
\item There are non-trivial (and sharp) bounds on the spectra in terms of familiar dimensions, see Propositions \ref{BB}  and \ref{BBL}, and also Corollaries \ref{BDconverge}  and \ref{BD0}.
\item The spectra are continuous in $\theta$, see Corollaries \ref{Con}  and \ref{Con2}.
\item The Assouad spectrum interpolates between upper box dimension and Assouad dimension.  In particular, as $\theta \to 0$ the spectrum always approaches the upper box dimension, and as $\theta \to 1$ the spectrum always approaches its maximal value, which is often the Assouad dimension. See Corollary \ref{BDconverge}, Corollary \ref{stayA}, and Proposition \ref{MONO}. 
\item The spectra are often, but not necessarily, monotonic, see Proposition \ref{MONO} and Section  \ref{NON-MONO2}.
\item The spectra have good distortion properties under bi-H\"older functions, which is in contrast to the Assouad dimension, see Proposition \ref{Holderthm}.
\item The Assouad and lower spectra are measureable, and of Baire class 2, when viewed as functions on the set of compact subsets of a metric space, see Theorems \ref{meas1}  and \ref{meas2}.
\item We analyse how the spectra behave under standard geometric operations such as unions, closures and products, see Propositions \ref{basic1}  and \ref{basic3}.
\end{enumerate}
In a subsequent paper \cite{fraseryu2} we compute the spectra explicitly for a range of important classes of fractal sets including: self-similar sets with overlaps, self-affine carpets, Mandelbrot percolation and Moran constructions.   In particular, the spectra can take on a range of different forms demonstrating the richness of the theory we introduce here. Although the main purpose of this article is to introduce, and conduct a thorough investigation of, our new dimension spectra, we also obtain several results as corollaries or bi-products of our work, which are not \emph{a priori} related to the spectra.  We summarise some of these  results here:
\begin{enumerate}
\item We provide new bi-H\"older distortion results for Assouad dimension, see Theorem \ref{Holderassouadcor}.  In particular, if the Assouad spectrum reaches the Assouad dimension, then we can give a bound on how the Assouad dimension distorts under a bi-H\"older map.  No such bounds exist for general sets.
\item We prove that the lower dimension is a Baire 2 function, which is sharp and improves upon a previous result of Fraser (see \cite[Question 4.6]{Fraser}) where it was only shown to be Baire 3 (see our Theorem \ref{meas3}).
\item We prove that sub-exponential spirals cannot be `unwound' to line segments via certain bi-H\"older functions, see Corollary \ref{spiralcor}.  This provides a natural extension to work of Fish and Paunescu concerning bi-Lipschitz unwinding \cite{spirals}, as well as classical unwinding theorems of Katznelson, Subhashis and Sullivan \cite{unwindspirals}.
\item We prove that a spiral with `monotonic winding' either has Assouad dimension 1 or 2, see Theorem \ref{assspiral}.
\end{enumerate}
Finally, in Section \ref{OpenSection} we collect several open questions and discuss possible directions for future work.

We begin by recalling the precise definition of the Assouad dimension, which serves to motivate our new definition.  Let $F \subseteq X$ where $X$ is a fixed metric space.  The \emph{Assouad dimension} of $F$ is defined by
\begin{eqnarray*}
\dim_\text{A} F &=&  \inf \bigg\{ \alpha \  : \   (\exists C>0) \, (\exists \rho>0) \, (\forall 0<r<R\leq \rho) \,  (\forall x \in F) \\ \\
&\,& \qquad \qquad \qquad   \qquad \qquad   N \big( B(x,R) \cap F ,r \big) \ \leq \ C \left(\frac{R}{r}\right)^\alpha \bigg\}.
\end{eqnarray*}
As described above, we will modify this definition by taking the infimum over the less restrictive condition that the scaling property only holds for scales $0<r<R\leq \rho$ satisfying a particular relationship.  For $\theta \in (0,1)$, we define
\begin{eqnarray*}
\dim_{\mathrm{A}}^\theta F &=& \inf \bigg\{ \alpha \  : \   (\exists C>0) \, (\exists \rho>0) \, (\forall 0<R\leq \rho) \,  (\forall x \in F) \\ \\
&\,& \qquad \qquad  \qquad \qquad   N \big( B(x,R) \cap F ,R^{1/\theta} \big) \ \leq \ C \left(\frac{R}{R^{1/\theta}}\right)^\alpha \bigg\}.
\end{eqnarray*}
We are particularly interested in the function $\theta \mapsto \dim_{\mathrm{A}}^\theta F$ which we refer to as the \emph{Assouad spectrum} (of $F$). For convenience we extend $\dim_\mathrm{A}^{\theta} F$ to $\theta\in(0,\infty)$ by setting $\dim_\mathrm{A}^{\theta} F=0$ for $\theta\geq 1$.

Of course there are many other ways to fix the relationship between the scales $r$ and $R$.  However, it turns out that if one wants to develop a rich theory, the most natural way to do this is what we propose here.  See the discussion in Section \ref{OpenSection} at the end of the paper for more details on this point.

The lower dimension, introduced by Larman \cite{Larman1, Larman2}, is the natural dual of the Assouad dimension.  We refer the reader to \cite{Fraser} for an in-depth discussion of the relationships and differences between these dimensions.  Let $F \subseteq X$ be as above.  The \emph{lower dimension} of $F$ is defined by
\newpage
\begin{eqnarray*}
\dim_\text{L} F &=&   \sup \bigg\{ \alpha \  : \   (\exists C>0) \, (\exists \rho>0) \, (\forall 0<r<R\leq \rho) \,  (\forall x \in F)\\ \\
&\,& \qquad \qquad \quad  \qquad  \qquad \qquad  N \big( B(x,R) \cap F,r \big) \ \geq \ C \left(\frac{R}{r}\right)^\alpha \bigg\}.
\end{eqnarray*}
Due to the local nature of this definition, it has many strange properties which may not be seen as desirable for a `dimension' to satisfy.  For example, it is not monotone as the presence of a single isolated point renders the lower dimension 0, and it may take the value 0 for an open subset of Euclidean space, see \cite[Example 2.5]{Fraser}.  One can modify the definition to get rid of these (perhaps) strange properties by defining the \emph{modified lower dimension} by
\[
\dim_\text{ML} F \ = \  \sup \left\{ \dim_\text{L} E \  : \  \emptyset \neq E \subseteq F \right\}.
\]
 For $\theta \in (0,1)$, we define
\begin{eqnarray*}
\dim_\text{L}^\theta F &=&   \sup \bigg\{ \alpha \  : \   (\exists C>0) \, (\exists \rho>0) \, (\forall 0<R\leq \rho) \,  (\forall x \in F) \\ \\
&\,& \qquad \qquad  \qquad \qquad   N \big( B(x,R) \cap F,R^{1/\theta} \big) \ \geq \ C \left(\frac{R}{R^{1/\theta}}\right)^\alpha \bigg\}.
\end{eqnarray*}
Again, we are particularly interested in the function $\theta \mapsto \dim_\text{L}^\theta F$ which we refer to as the \emph{lower spectrum} (of $F$). As before, we extend $\dim_\text{L}^{\theta} F$ to $\theta\in(0,\infty)$ by setting $\dim_\text{L}^{\theta} F=0$ for $\theta\geq 1$.  We can also modify this definition to force it to be monotone and to take on the ambient spatial dimension for open sets.  The \emph{modified lower spectrum} (of $F$) is defined by
\[
\dim_\text{ML}^\theta F \ = \  \sup \left\{ \dim_\text{L}^\theta E \  : \  \emptyset \neq E \subseteq F \right\}.
\]
The key motivation behind these new definitions is that the geometric information provided by the Assouad and lower dimensions is too coarse.  We gain more information by understanding how the inhomogeneity depends on the scales one is considering.  Alternative approaches to getting more out of these dimensions are possible.  For example, Fraser and Todd \cite{FraserTodd} recently considered a quantitative analysis of the Assouad dimension where they looked to understand how inhomogeneity varies in space, i.e. as one changes the point $x \in F$ around which one is trying to cover the set.  They found that for some natural examples this inhomogeneity could be described by a \emph{Large Deviations Principle}.  In a certain sense our approach is dual to that of \cite{FraserTodd} in that we put restrictions on \emph{scale} but still maximise over \emph{space}, whereas in \cite{FraserTodd} restrictions were put on \emph{space}, but the quantities were still maximised over all \emph{scales}.

\section{Notation and preliminaries}

 Here we summarise some notation which we will use throughout the paper.  For real-valued functions $f,g$, we write $f(x)\lesssim g(x)$ to mean that there exists a universal constant $M>0$, independent of $x$, such that $f(x)\leq Mg(x)$. Some readers may be more familiar with the notation $f(x) = O(g(x))$, which is sometimes more convenient and means the same thing.   Similarly, $f(x)\gtrsim g(x)$ means that $f(x)\geq Mg(x)$ with a universal constant independent of $x$.  If both $f(x)\lesssim g(x)$ and $f(x)\gtrsim g(x)$, then we write $f(x)\asymp g(x)$.  Generally one should think of $x$ as being the tuple consisting of all variables in the expression $f(x)$.  Usually $x$ will be a length scale but could sometimes also incorporate points in the metric space in question or  other independent length scales.

For a real number $a$, we write $a^+$ to denote a real number that is strictly larger than $a$ but can be chosen as close to $a$ as we wish. Similarly, we write $a^-$ to denote a real number that is strictly less than $a$ but can be chosen as close to $a$ as we wish. 

For real numbers $a,b$, we write $a \wedge b$ for the minimum of the two numbers and $a \vee b$ for the maximum.  Also, for a non-negative real number $x \geq 0$, we write $[x]$ for the integer part of $x$.

The Assouad and lower dimensions are closely related to the upper and lower box dimension.  The \emph{upper} and \emph{lower box dimensions} of a totally bounded set $F$ are defined by
\[
\overline{\dim}_{\mathrm{B}}F  \ = \  \limsup_{R \to 0} \   \frac{\log  N \big( F, R\big) }{-\log R} \qquad \text{ and } \qquad   \underline{\dim}_\mathrm{B} F  \ = \  \liminf_{R \to 0} \   \frac{\log  N \big( F, R\big) }{-\log R}
\]
and if the upper and lower box dimensions coincide we call the common value the \emph{box dimension} of $F$ and denote it by $\dim_\mathrm{B} F$.  We refer the reader to \cite[Chapter 3]{Falconer} for more details on the upper and lower box dimensions and their basic properties.  In particular we note the following general relationships which hold for any totally bounded set $F$:

\[
 \dim_\text{L} F \   \leq   \  \dim_\text{ML} F  \   \leq   \   \underline{\dim}_\text{B} F   \   \leq   \  \overline{\dim}_\text{B} F  \   \leq   \ \dim_\text{A} F,
\]
see \cite{Fraser, Larman1,Larman2}.  The upper and lower box dimensions will play an important role in our analysis.

Unlike the Assouad and lower dimensions, we can give simple explicit formulae for the dimension spectra.  Indeed, it follows immediately from the definitions that
\[
\dim_{\mathrm{A}}^\theta F  \ = \  \limsup_{R \to 0} \  \sup_{x \in F} \, \frac{\log  N\big(B(x,R) \cap F,R^{1/\theta}\big) }{(1-1/\theta)\log R} 
\]
and
\[
\dim_\text{L}^\theta F  \ = \  \liminf_{R \to 0} \  \inf_{x \in F} \, \frac{\log  N\big(B(x,R) \cap F, R^{1/\theta}\big) }{(1-1/\theta)\log R}.
\]
As with the definitions of Assouad and lower dimension, as well as the upper and lower box dimensions, the definition of $N(\cdot,r)$ may be replaced with a number of related concepts without altering any of the definitions.  For example, if working in $\mathbb{R}^d$ we could use the number of $r$-cubes in an $r$-mesh which intersect the given set.  Another possibility is to let $N(E,r)$ be the maximal cardinality of an $r$-packing of $E$, where an $r$-packing is a collection of closed pairwise disjoint balls of radius $r$ with centres in $E$.  Also, using the explicit formulae given above, we see that letting $R\to 0$ through an exponential sequence of scales, such as $2^{-k}$ ($k \in \mathbb{N}$), yields the same limits.  We leave it to the reader to show that these variations lead to the same dimensions and spectra and refer to \cite[Chapter 3]{Falconer}) for more details.  It is often useful to adopt these different definitions of $N(\cdot,r)$, in particular when we consider measureability properties in Section \ref{BorelSect}.

\section{Analytic properties and general bounds}

Our first proposition gives general (and sharp) bounds on the Assouad spectrum in terms of the Assouad and box dimensions.

\begin{prop} \label{BB}
Let $F$ be a totally bounded set.  Then for all $\theta \in (0,1)$ we have
\[
\overline{\dim}_{\mathrm{B}}F \  \leq \  \dim_{\mathrm{A}}^\theta F  \ \leq  \    \frac{\overline{\dim}_{\mathrm{B}}F}{{1-\theta} }  \wedge  \dim_\mathrm{A} F .
\]
\end{prop}

\begin{proof}

We will write $B$ for the upper box dimension of $F$.  First of all there is a clear upper bound holding for any $x\in F$ and small enough $R$:

$$N(B(x,R)\cap F,R^{1/\theta})\leq N(F,R^{1/\theta})\lesssim R^{-B^+/\theta}.$$

This implies that
\begin{equation} \label{estimateBB1}
\sup_{x\in F}N(B(x,R),R^{1/\theta})\leq N(F,R^{1/\theta})\lesssim R^{-B^+/\theta}.
\end{equation}

Whenever we have a covering of $F$ by $R$-balls, if we further cover each $R$-ball with $R^{1/\theta}$-balls then we get a cover of $F$ by $R^{1/\theta}$-balls and an upper bound for $N(F,R^{1/\theta})$.  We can cover $F$ with $N(F,R)$ $R$-balls, and all those $R$-balls can be covered by at most $\sup_{x\in F}N(B(x,R)\cap F,R^{1/\theta})$ many $R^{\frac{1}{\theta}}$-balls, therefore $$\sup_{x\in F}N(B(x,R)\cap F,R^{1/\theta})N(F,R)\geq N(F,R^{1/\theta})$$
and so
$$\frac{N(F,R^{1/\theta})}{N(F,R)}\leq \sup_{x\in F}N(B(x,R)\cap F,R^{1/\theta})\leq N(F,R^{1/\theta}).$$
Since $N(F,R)\lesssim R^{-B^+}$ for all small enough $R$ and $N(F,R)\gtrsim R^{-B^-}$ for infinitely many $R\rightarrow 0$ we have that
\begin{equation} \label{estimateBB2}
\sup_{x\in F}N(B(x,R)\cap F,R^{1/\theta})\gtrsim N(F,R^{1/\theta})R^{B^+}\gtrsim R^{-B^-/\theta+B^+}
\end{equation}
holds for a sequence of $R \to 0$.  It now follows from  (\ref{estimateBB1}) and  (\ref{estimateBB2}) that
$$\frac{B^-/\theta-B^+}{1/\theta-1} \ \leq \  \dim_\mathrm{A}^{\theta} F \ \leq \ \frac{B^+/\theta}{1/\theta-1}  \ = \ \frac{B^+}{1-\theta}.$$
Finally, since $\dim_\text{A} F$ is a trivial upper bound for $\dim_\text{A}^{\theta} F$ for all $\theta\in (0,1)$, the desired conclusion follows.
\end{proof}

The above estimates show that if the upper box and Assouad dimensions of a set coincide, then the Assouad spectrum is constantly equal to the common value for all $\theta \in (0,1)$.  Such sets are highly homogeneous and therefore it is not surprising that the Assouad spectrum yields no new information.  Fortunately, sets with distinct upper box and Assouad dimensions abound and we will focus on such examples.  In this case, the above estimates show that, in some sense, the Assouad spectrum \emph{must} yield finer information than the upper box and Assouad dimensions alone. Indeed, the only way the spectrum can be constant is if it is constantly equal to the upper box dimension, but such behaviour would be quite striking since the definition is more similar to the Assouad dimension than the upper box dimension. In such cases, the Assouad dimension is not `seen' by any $\theta$ and this shows that to obtain the Assouad dimension, one must use  a complicated collection of pairs $(R,r)$ without any clear exponential relationship.

We also note that these general bounds are sharp.  In particular, we show that the upper bound is always attained  for a natural family of decreasing sequences, see Section \ref{SequencesSect}.  We also give examples where the \emph{lower} bound is always attained (and the Assouad dimension is strictly larger than the upper box dimension), see Example \ref{lowersharpe}.  We also note that for many natural examples the spectrum lies strictly between these upper and lower bounds.  For example, in the sequel \cite{fraseryu2} we show that the spectra necessarily lie strictly between the general upper and lower bounds for  the self-affine carpets studied by Bedford and McMullen (provided the construction has non-uniform fibres).

Letting $\theta\rightarrow 0$ in the previous result we obtain the following corollary:

\begin{cor}  \label{BDconverge}
For any totally bounded set $F$, we have
\[
\dim_{\mathrm{A}}^\theta F \to \overline{\dim}_{\mathrm{B}}F 
\]
as $\theta \to 0$.
\end{cor}

We note that one cannot say anything as succinct about the limit of $\dim_{\mathrm{A}}^\theta F$ as $\theta \to 1$, but it is often the case that $\dim_{\mathrm{A}}^\theta F = \dim_\text{A} F$ in some (sometimes large) interval to the left of $\theta = 1$.

Proposition \ref{BB} has the following immediate corollary, which at first sight looks surprising since the definition of the Assouad spectrum does not appear to depend so sensitively on the upper box dimension.

\begin{cor} \label{BD0}
For any totally bounded set $F$ with  $\overline{\dim}_{\mathrm{B}}F=0$, we have
\[
\dim_{\mathrm{A}}^\theta F  = 0
\]
for all $\theta \in (0,1)$.
\end{cor}

We now move towards analytic properties of the spectra.  Our first result is a technical regularity observation, which has some useful consequences.

\begin{prop}\label{Abounds}
For any set $F$ and $0<\theta_1<\theta_2<1$ we have
\[
\dim_\mathrm{A}^{\theta_2} F \left(\frac{\frac{1}{\theta_2}-1}{\frac{1}{\theta_1}-1}\right) \ \leq \  \dim_\mathrm{A}^{\theta_1} F  \ \leq  \ \dim_\mathrm{A} F \left(\frac{\frac{1}{\theta_1}-\frac{1}{\theta_2}}{\frac{1}{\theta_1}-1} \right) +\dim_\mathrm{A}^{\theta_2} F \left(\frac{\frac{1}{\theta_2}-1}{\frac{1}{\theta_1}-1} \right).
\]
\end{prop}

\begin{proof}

Following similar ideas as in the proofs above, for $0<\theta_1<\theta_2<1$, we have for any $R>0$
$$\sup_{x\in F}N(B(x,R)\cap F,R^{1/\theta_1})\geq \sup_{x\in F}N(B(x,R)\cap F,R^{1/\theta_2}).$$

Within $F$ we can cover any $R^{\frac{1}{\theta_2}}$-ball with 
\[
\lesssim \ \left(\frac{R^{\frac{1}{\theta_2}}}{R^{\frac{1}{\theta_1}}}\right)^{\dim_\mathrm{A} F^+}
\]
many $R^{\frac{1}{\theta_1}}$-balls.  Then for small enough $R>0$ we have
$$\left(\frac{R^{\frac{1}{\theta_2}}}{R^{\frac{1}{\theta_1}}}\right)^{\dim_\text{A} F^+}\sup_{x\in F}N(B(x,R)\cap F,R^{1/\theta_2}) \ \geq  \ \sup_{x\in F}N(B(x,R)\cap F,R^{1/\theta_1}).$$
Therefore,
\begin{eqnarray*}
R^{\dim_\text{A} F^+(\frac{1}{\theta_2}-\frac{1}{\theta_1})}\sup_{x\in F}N(B(x,R)\cap F,R^{1/\theta_2}) & \geq & \sup_{x\in F}N(B(x,R)\cap F,R^{1/\theta_1}) \\ \\ 
 & \geq & \sup_{x\in F}N(B(x,R)\cap F,R^{1/\theta_2})
\end{eqnarray*}
Also notice that for a sequence of $R\to 0$ we have
$$\sup_{x\in F}N(B(x,R)\cap F,R^{1/\theta_2})\gtrsim R^{(1-1/\theta_2)\dim_\mathrm{A}^{\theta_2} F^-}$$
as well as for any sufficiently small $R>0$ we have
$$\sup_{x\in F}N(B(x,R)\cap F,R^{1/\theta_2})\lesssim R^{(1-1/\theta_2)\dim_\mathrm{A}^{\theta_2} F^+}.$$
The desired bounds then follow immediately from the definitions.
\end{proof}

The bounds in Proposition \ref{Abounds} have some very useful consequences, such as  continuity of the spectrum.

\begin{cor}\label{Con}
For any $0<\theta_1 \leq \theta_2<1$ we have
\[
\lvert \dim_\mathrm{A}^{\theta_1} F - \dim_\mathrm{A}^{\theta_2} F \rvert \ \leq \ \frac{\dim_\mathrm{A}F}{\theta_2(1-\theta_1)} \lvert{\theta_1} - {\theta_2} \rvert.
\]
In particular, the function $\theta \mapsto \dim_{\mathrm{A}}^\theta F$ is continuous in $\theta \in (0,1)$ and for any $\varepsilon>0$, the function $\theta \mapsto \dim_{\mathrm{A}}^\theta F$ is Lipschitz  on the interval $[\varepsilon, 1-\varepsilon]$.
\end{cor}

\begin{proof}
This follows immediately from the bounds presented in Proposition \ref{Abounds}.
\end{proof}

Continuity of the dimension spectra is a useful property, especially when dealing with random fractals such as Mandelbrot percolation: a \emph{continuous} function is determined by its values on a \emph{countable}  dense set.  We point out that the spectra are not any more regular than continuous, as for most of our examples the spectra exhibit phase transitions where they fail to be differentiable.  

Another useful consequence of Proposition \ref{Abounds} is that `if the spectrum reaches the Assouad dimension of $F$, then it stays there'.

\begin{cor}\label{stayA}
If for some $\theta \in (0,1)$, we have $\dim_{\mathrm{A}}^\theta F = \dim_\mathrm{A} F$, then
\[
\dim_\mathrm{A}^{\theta'} F = \dim_\mathrm{A} F
\]
for all $\theta' \in [ \theta,1)$.
\end{cor}
\begin{proof}
Starting with the right hand inequality from Proposition \ref{Abounds}, assume that $\dim_\mathrm{A}^{\theta_1} F = \dim_\mathrm{A} F$.  This immediately gives that $\dim_\mathrm{A}^{\theta_2} F \geq \dim_\mathrm{A} F$ which, together with Proposition \ref{BB}, proves the result.
\end{proof}

Another natural question concerns monotonicity.  Indeed, all of the `natural' examples we consider have monotone spectra, i.e. the spectrum is non-decreasing in $\theta$.  Surprisingly, this is not always the case.  We exhibit this by constructing an example in Section \ref{NON-MONO2}.  The following result shows that one does have some sort of `quasi-monotonicity', however.

\begin{prop}\label{MONO}
For any $F$ and $0<\theta_1 < \theta_2 < 1$ we have
$$\dim_\mathrm{A}^{\theta_1} F \ \leq \  \left(\frac{1-\theta_2}{1-\theta_1} \right) \dim_\mathrm{A}^{\theta_2} F \, + \, \left( \frac{\theta_2-\theta_1}{1-\theta_1} \right) \dim_\mathrm{A}^{\theta_1/\theta_2} F.$$
In particular, by setting $\theta_2=\sqrt{\theta_1}$, we have
\[
\dim_\mathrm{A}^{\theta_1}F \  \leq   \ \dim_\mathrm{A}^{\sqrt{\theta_1}}F
\]
for any $\theta_1\in (0,1)$.  Furthermore,  this implies that for any $\theta\in(0,1)$, we can find $\theta'$ arbitrarily close to $1$ such that $\dim_\mathrm{A}^{\theta'}F\geq \dim_\mathrm{A}^{\theta}F$.
\end{prop}

\begin{proof}
Fix $0<\theta_1<\theta_2<1$ and notice that $0<\frac{\theta_1}{\theta_2}<1$.  For sufficiently small $R>0$, we  have:
$$\sup_{x\in F}N(B(x,R)\cap F,R^{1/\theta_i})\lesssim \left(R^{1-\frac{1}{\theta_i}}\right)^{{\dim_\mathrm{A}^{\theta_i} F}^+}$$ 
for $i = 1,2$.  We can also find infinitely many $R\to 0$ such that:
$$\sup_{x\in F}N(B(x,R)\cap F,R^{1/\theta_i})\gtrsim \left(R^{1-\frac{1}{\theta_i}}\right)^{{\dim_\mathrm{A}^{\theta_i} F}^-}.$$
Let $R'=R^{\theta_2/\theta_1}$, and observe that $(R')^{1/\theta_2}=R^{1/\theta_1}$. We can cover any $R$ ball with at most $\lesssim (R/R')^{\dim_\mathrm{A}^{\theta_1/\theta_2}F^+}$ balls with radius $R'$ for small enough $R$.

Then we need no more than $\sup_{x\in F}N(B(x,R')\cap F,R^{1/\theta_1})$ balls with radius $R^{\theta_1}$ to cover any $R'$-ball.

Given an arbitrary $R$-ball, first cover it with $R'$-balls, and then cover those $R'$-balls by $R^{\theta_1}$-balls using optimal covers as indicated above.  This yields  
$$  \sup_{x\in F}N(B(x,R)\cap F,R^{1/\theta_1})   \ \leq \     \left(\frac{R}{R'}\right)^{\dim_\mathrm{A}^{\theta_1/\theta_2} F^+} \sup_{x\in F}N(B(x,R')\cap F,R^{1/\theta_1}) $$
but since $(R')^{1/\theta_2}=R^{1/\theta_1}$ we have
$$\sup_{x\in F}N(B(x,R')\cap F,R^{1/\theta_1})\lesssim \left((R')^{1-\frac{1}{\theta_2}}\right)^{{\dim_\mathrm{A}^{\theta_2} F}^+}$$ for all $R'$ small enough and also
$$\sup_{x\in F}N(B(x,R)\cap F,R^{1/\theta_1})\gtrsim \left(R^{1-\frac{1}{\theta_1}}\right)^{{\dim_\mathrm{A}^{\theta_1} F}^-}$$ for some arbitrarily small $R$.

Therefore for some arbitrarily small $R$ we get:
$$
\left(\frac{R'}{R}\right)^{\dim_\mathrm{A}^{\theta_1/\theta_2} F^+} \left(R^{1-\frac{1}{\theta_1}}\right)^{{\dim_\mathrm{A}^{\theta_1} F}^-}
 \ \lesssim  \ \left((R')^{1-\frac{1}{\theta_2}}\right)^{\dim_\mathrm{A}^{\theta_2} F^+}$$
and so replacing $R'$ by  $R^{\theta_2/\theta_1}$, taking logs, and dividing through by $\log R$ yields
\[
\left( \frac{\theta_2}{\theta_1} - 1\right) \dim_\mathrm{A}^{\theta_1/\theta_2}F^+\ + \ \left( 1 - \frac{1}{\theta_1}\right) \dim_\mathrm{A}^{\theta_1}F^- \ \geq \ \frac{\theta_2}{\theta_1}\left(1- \frac{1}{\theta_2} \right) \dim_\mathrm{A}^{\theta_2}F^+.
\]
This, in turn, yields
\[
\dim_\mathrm{A}^{\theta_1} F \ \leq \  \left(\frac{1-\theta_2}{1-\theta_1} \right) \dim_\mathrm{A}^{\theta_2} F \, + \, \left( \frac{\theta_2-\theta_1}{1-\theta_1} \right) \dim_\mathrm{A}^{\theta_1/\theta_2} F
\]
as required.
\end{proof}

\begin{rem}
The strategy of the above proof is to cover a large ball with middle sized balls and further cover the middle sized balls with smaller balls.  This can be generalized to arbitrarily many levels of covering to obtain more general results.

We can cover an $R$-ball with $R^{\frac{1}{\theta_n}}$-balls, then each of these balls with $R^{\frac{1}{\theta_{n-1}}}$-balls and so on and then applying the same proof strategy as above we end up with the following inequality: for 
$$0<\theta_1<\theta_2< \cdots <\theta_n<1$$
we have
$$
\dim_\mathrm{A}^{\theta_1} F
 \ \leq \ \left(\frac{1-\theta_n}{1-\theta_1} \right)
\dim_\mathrm{A}^{\theta_n} F \ + \  \sum_{i=2}^{n} 
\left(\frac{\theta_i-\theta_{i-1}}{1-\theta_1} \right)
\dim_\mathrm{A}^{\theta_{i-1}/\theta_{i}} F
   $$

By setting $\theta_i=\theta_1^{\frac{n-i+1}{n}}$ we end up with:

$$\dim_\mathrm{A}^{\theta_1} F \ \leq  \ \dim_\mathrm{A}^{\sqrt[n]{\theta_1}} F$$ 

for any $\theta_1\in (0,1)$ and any natural number $n$, which is a slightly stronger result.
\end{rem}

We will now discuss the analogous properties for the lower (and modified lower)  spectrum.

\begin{prop} \label{BBL}
Let $F$ be a totally bounded set.  Then for all $\theta \in (0,1)$ we have
\[
\dim_\mathrm{L} F  \  \leq \  \dim_\mathrm{L}^\theta F  \ \leq  \     \underline{\dim}_\mathrm{B} F
\]
and
\[
\dim_\mathrm{ML} F  \  \leq \  \dim_\mathrm{ML}^\theta F  \ \leq  \     \underline{\dim}_\mathrm{B} F
\]
\end{prop}

\begin{proof}
First note that it follows immediately from the definitions that $\dim_\mathrm{L} F  \  \leq \  \dim_\mathrm{L}^\theta F $ and therefore also $\dim_\mathrm{ML} F  \  \leq \  \dim_\mathrm{ML}^\theta F $.

We will now prove the upper bounds and during the proof we will write $b$ for the lower box dimension of $F$.  Fix $\theta \in (0,1)$ and $R \in (0,1)$.  Let $M(E,r)$ denote the largest possible cardinality of an $r$-packing of a set $E$ by closed balls of radius $r$. Take an optimal $2R$-packing of $F$ by closed balls and then inside each of these balls construct an optimal $R^{1/\theta}$-packing of the smaller ball centered at the same point but with radius $R$.  The resulting $R^{1/\theta}$-balls are centered in $F$ and are pairwise disjoint and, therefore, one obtains an $R^{1/\theta}$-packing of $F$ by more than
\[
M(F,2R) \, \inf_{x \in F} M\left(B(x,R), \, R^{1/\theta} \right)
\]
balls.  This yields
\[
 \inf_{x \in F} M\left(B(x,R), \, R^{1/\theta} \right)  \ \leq  \ \frac{M(F,R^{1/\theta})}{M(F,2R)} \ \lesssim \ \frac{R^{-b^+/\theta}}{R^{-b^-}} \ = \ R^{-(b^+/\theta-b^-)}
\]
for infinitely many $R \to 0$. It then follows from the definitions
\[
\dim_\mathrm{L}^\theta F  \ \leq  \ \frac{b^+/\theta-b^-}{1/\theta-1}
\]
from which the desired upper bound follows.  This also passes to the modified lower spectrum, completing the proof.
\end{proof}

Since the lower dimension is bounded above by the modified lower dimension (and the Hausdorff dimension if the set is compact, see \cite{Larman1})  it is natural to ask if this is also (uniformly) true for the lower spectrum, i.e., if the upper bounds in Proposition \ref{BBL} can be improved?  Perhaps surprisingly, this is not the case.  In particular, for self-affine carpets the lower spectrum approaches the box dimension as $\theta \to 0$, see our sequel paper \cite{fraseryu2}.

\begin{thm}  \label{Con2}
The functions $\theta \mapsto \dim_\mathrm{L}^{\theta} F$ and $\theta \mapsto \dim_\mathrm{ML}^{\theta} F$  are continuous in $\theta \in (0,1)$.  Moreover, they are Lipschitz on any closed subinterval of $(0,1)$.
\end{thm}

\begin{proof}
For any $0<R<1$ and $0<\theta_1<\theta_2<1$ we have $R^{1/\theta_1}<R^{1/\theta_2}$, therefore it is clear that for any $0<R<1$:
$$\inf_{x\in F}N(B(x,R)\cap F,R^{1/\theta_2})
\leq \inf_{x\in F}N(B(x,R)\cap F,R^{1/\theta_1}).$$
Now notice that, 
$$  \inf_{x\in F}N(B(x,R)\cap F,R^{1/\theta_1}) \lesssim \inf_{x\in F}N(B(x,R)\cap F,R^{1/\theta_2})\left(\frac{R^{1/\theta_2}}{R^{1/\theta_1}}\right)^{\dim_\mathrm{A} F^+}. $$
This is because we may cover at least one $R$-ball with
$$ \lesssim \inf_{x\in F}N(B(x,R)\cap F,R^{1/\theta_2})\left(\frac{R^{1/\theta_2}}{R^{1/\theta_1}}\right)^{\dim_\mathrm{A} F^+}$$
balls of radius $R^{1/\theta_1}$, and therefore this number is no smaller than $\inf_{x\in F}N(B(x,R)\cap F,R^{1/\theta_1})$.  The above two inequalities imply that for infinitely many $R\to 0$ we have
$$R^{(1-1/\theta_2)\dim_\mathrm{L}^{\theta_2} F^-}
\lesssim R^{(1-1/\theta_1)\dim_\mathrm{L}^{\theta_1} F^+}$$
and also for  infinitely many $R\to 0$ we have
$$R^{(1-1/\theta_1)\dim_\mathrm{L}^{\theta_1}F^-}
\lesssim R^{(1-1/\theta_2)\dim_\mathrm{L}^{\theta_2} F^+}R^{\dim_\mathrm{A} F^+(1/\theta_2-1/\theta_1)}.$$
It follows that
\begin{eqnarray*}
\left(1-\frac{1}{\theta_2}\right)\dim_\textrm{L}^{\theta_2}F+\dim_\textrm{A} F\left(\frac{1}{\theta_2}-\frac{1}{\theta_1}\right) &\leq& \left(1-\frac{1}{\theta_1}\right)\dim_\textrm{L}^{\theta_1}F \\ \\
&\leq&  \left(1-\frac{1}{\theta_2}\right)\dim_\textrm{L}^{\theta_2}F
\end{eqnarray*}
Dividing through by $(1-1/\theta_1)$ and then letting $\theta_1 \nearrow \theta_2$ establishes lower semicontinuity of $\theta \mapsto \dim_\textrm{L} F$ at $\theta_2$ and letting $\theta_2 \searrow \theta_1$ establishes upper semicontinuity of $\theta \mapsto \dim_\textrm{L} F$ at $\theta_1$.  Since $\theta_1$ and $\theta_2$ are arbitrary the desired continuity follows.

The above discussion holds for any metric space $F$, and in particular for any subspace $E\subseteq F$ we have
\begin{eqnarray*}
\left(1-\frac{1}{\theta_2}\right)\dim_\textrm{L}^{\theta_2}E+\dim_\textrm{A} E\left(\frac{1}{\theta_2}-\frac{1}{\theta_1}\right)&\leq&  \left(1-\frac{1}{\theta_1}\right)\dim_\textrm{L}^{\theta_1}E \\ \\
&\leq& \left(1-\frac{1}{\theta_2}\right)\dim_\textrm{L}^{\theta_2}E
\end{eqnarray*}
Taking the supremum over all $E\subseteq F$ throughout, we get
\begin{eqnarray*}
\left(1-\frac{1}{\theta_2}\right)\dim_\textrm{ML}^{\theta_2}F-\dim_\textrm{A} F\left(\frac{1}{\theta_1}-\frac{1}{\theta_2}\right)&\leq& \left(1-\frac{1}{\theta_1}\right)\dim_\textrm{ML}^{\theta_1}F \\ \\
&\leq& \left(1-\frac{1}{\theta_2}\right)\dim_\textrm{ML}^{\theta_2}F
\end{eqnarray*}
and therefore the modified lower spectrum is also continuous.  Finally, the fact that the lower spectrum and modified lower spectrum are Lipschitz on any closed subinterval of $(0,1)$ also follows immediately by applying the above bounds.
\end{proof}

\newpage

\section{Geometric properties}

In this section we investigate how the various dimension spectra are affected by standard geometric operations such as products, unions, and images under H\"older continuous maps.

It is clear that the spectra satisfy the following properties and we leave the proofs to the reader.

\begin{prop}[Closure, monotonicity, and finite stability] \label{basic1} \label{basic2}
\hspace{1mm} \\

\begin{enumerate}
\item For any set $F$ in a metric space and any $\theta \in (0,1)$, we have:
$$\dim_\mathrm{A}^{\theta} F \ = \ \dim_\mathrm{A}^{\theta} \overline{F}$$
$$\dim_\mathrm{L}^{\theta} F \ = \ \dim_\mathrm{L}^{\theta} \overline{F}.$$
\item For any $F'\subseteq F$ and any $\theta \in (0,1)$, we have:
$$\dim_\mathrm{A}^{\theta} F' \ \leq \ \dim_\mathrm{A}^{\theta} F$$
$$\dim_\mathrm{ML}^{\theta} F' \ \leq \ \dim_\mathrm{ML}^{\theta} F.$$
\item For any finite collection of sets  $\{F_i\}_{i =1}^n$ we have, for all $\theta\in(0,1)$, 
\[
\dim_\mathrm{A}^{\theta} \left( \bigcup_{i=1}^{n}F_i \right) \ = \  \max_{i=1,2, \dots, n}\dim_\mathrm{A}^{\theta} F_i.
\]
\end{enumerate}
\end{prop}

Interestingly, the modified lower dimension and modified lower spectrum are not stable under taking closure as the following example illustrates.

\begin{example}
Let
\[
X \ = \ \left\{ (p/q, 1/q) \ : \ p, q \in \mathbb{N}^+,  \, p \leq q, \, \text{\emph{gcd}}(p,q) = 1 \right\}  \ \subseteq  \ [0,1]^2
\]
and observe that every point $x \in X$ is isolated and therefore any subset of $X$ has an isolated point.  This implies that for any $\theta \in (0,1)$
\[
\dim_\mathrm{ML}^{\theta} X \ = \ \dim_\mathrm{ML} X \ = \ 0.
\]
However, $[0,1] \times \{0 \} \subseteq \overline{X}$ and so 
\[
\dim_\mathrm{ML}^{\theta} \overline{X} \ = \ \dim_\mathrm{ML} \overline{X} \ = \ 1.
\]
\end{example}

Clearly the Assouad spectrum is not stable under countable unions.  For example  $\mathbb{Q}\cap [0,1]$ is a countable union of point sets, all of which have Assouad spectrum constantly equal to $0$, but $\mathbb{Q}\cap [0,1]$ has Assouad spectrum constantly equal to $1$ by the closure property.  The  lower spectrum is not even stable under finite unions: consider the union of $[0,1] \cup \{2\}$ and $\{0\} \cup [1,2]$.  One can say more if the sets in the union are properly separated.

\begin{prop}[Unions of properly separated sets]  \label{basic2sep}
Let $E, F$ be `properly separated' subsets of a metric space $(X,d)$, i.e. sets such that
\[
\inf_{x \in E, y \in F}  \, d(x,y) >0.
\]
Then,
\[
\dim_\mathrm{ML}E \cup F \  = \   \dim_\mathrm{ML}E  \, \vee \,     \dim_\mathrm{ML}F
\]
and, for all $\theta\in(0,1)$, 
\[
\dim_\mathrm{ML}^{\theta}E \cup F \  = \   \dim_\mathrm{ML}^{\theta}E  \, \vee \,     \dim_\mathrm{ML}^{\theta}F
\]
and
\[
\dim_\mathrm{L}^{\theta}E \cup F \  = \   \dim_\mathrm{L}^{\theta}E  \, \wedge \,     \dim_\mathrm{L}^{\theta}F.
\]
Moreover, these results extend to arbitrary finite unions of pairwise `properly separated' sets where the maximum/minimum is taken over all sets in the union.
\end{prop}

\begin{proof}
The argument for the lower spectrum is similar to \cite[Theorem 2.2]{Fraser} and is omitted.  For the modified lower dimension and modified lower spectrum, the proof is straightforward and we only briefly give the modified lower dimension argument.  The lower bound ($\geq$) follows from monotonicity.  For the upper bound, we have
\begin{eqnarray*}
\dim_\mathrm{ML}E \cup F &=& \sup_{\emptyset \neq Z \subseteq E \cup F} \dim_\mathrm{L} ( Z \cap E) \cup (Z \cap F) \\ \\
&=&   \sup_{\emptyset \neq Z \subseteq E \cup F} \bigg( \dim_\mathrm{L} ( Z \cap E) \wedge   \dim_\mathrm{L} (Z \cap F) \bigg) \\ \\
& \leq& \left(\sup_{\emptyset \neq Z \subseteq E } \dim_\mathrm{L}Z\right) \  \vee  \   \left(\sup_{\emptyset \neq Z \subseteq F }  \dim_\mathrm{L} Z \right)  \\ \\
&=& \dim_\mathrm{ML}E  \, \vee \,     \dim_\mathrm{ML}F
\end{eqnarray*}
as required.  Note that we used the fact that the lower dimension of the union of two properly separated sets is given by the minimum of the individual dimensions, which is provided in \cite[Theorem 2.2]{Fraser}. We also adopt the convention that $\dim_\mathrm{L} \emptyset = +\infty$.
\end{proof}

There are many results in dimension theory related to how the dimension of a product space depends on the dimensions of the marginals.  A common phenomenon is that dimensions are best considered in pairs and the following standard formula has been verified for many `dimension pairs' $\dim$ and $\mathrm{Dim}$:
\begin{eqnarray*}
\dim X \, + \, \dim Y \ \leq \ \dim (X \times Y)  \  \leq \  \dim X \, + \, \mathrm{Dim} \,  Y  &\leq&  \mathrm{Dim} \,  (X \times Y) \\ 
 &\leq& \mathrm{Dim} \,  X \, + \, \mathrm{Dim} \, Y.
\end{eqnarray*}
Such examples include Hausdorff and packing dimension, see Howroyd \cite{products}; lower and upper box dimension; and lower and Assouad dimension.  For recent works on such product formulae see \cite{Fraser,Olson1,Olson2}.  We show below that the Assouad and lower spectra give rise to a continuum of `dimension pairs'. 

There are many natural `product metrics' to impose on the product $X \times Y$ of metric spaces $(X,d_X)$ and $(Y,d_Y)$, with a natural choice being the sup metric $d_{X \times Y}$ on $X \times Y$ defined by
\[
d_{X \times Y}\big((x_1,y_1), (x_2,y_2)\big) = d_X(x_1,x_2) \vee  d_Y(y_1,y_2).
\]
In particular, this metric is compatible with the product topology and bi-Lipschitz equivalent with many other commonly used product metrics, such as those induced by $p$ norms.

\begin{prop}[Products] \label{basic3}
Let $E,F$ be metric spaces and equip the product $E \times F$ with any suitable product metric.  For any $\theta \in (0,1)$ we have
\[
 \dim_\mathrm{ML}^{\theta}E + \dim_\mathrm{A}^{\theta}F \leq \dim_\mathrm{A}^{\theta}(E \times F)\leq \dim_\mathrm{A}^{\theta}E+\dim_\mathrm{A}^{\theta}F
\]
\[
 \dim_\mathrm{L}^{\theta}E + \dim_\mathrm{L}^{\theta}F \leq \dim_\mathrm{L}^{\theta}(E \times F)\leq \dim_\mathrm{L}^{\theta}E+\dim_\mathrm{A}^{\theta}F
\]
and
\[
 \dim_\mathrm{ML}^{\theta}E + \dim_\mathrm{ML}^{\theta}F \leq \dim_\mathrm{ML}^{\theta}(E \times F)\leq \dim_\mathrm{ML}^{\theta}E+\dim_\mathrm{A}^{\theta}F.
\]
\end{prop}

\begin{proof}
For the purposes of this proof we use the sup metric on the product space.  In particular this means that the product of two covering sets of diameter $r$ is a set of diameter $r$ and so covers of parts of $E$ and $F$ can be easily combined to provide covers of the corresponding parts of $E \times F$.   Let  $P_E$ and $P_F$ denote the  projection on to $E$ and $F$ respectively. Then clearly for any $R>0$ and $x \in E \times F$ we have
\begin{eqnarray*}
N(B(x,R)\cap E\times F,R^{1/\theta}) &\leq & \sup_{y\in E}N(B(y,R)\cap E,R^{1/\theta})\sup_{z\in F}N(B(z,R)\cap F,R^{1/\theta}) \\ 
&\lesssim & R^{(1-1/\theta)(\dim_\mathrm{A}^\theta E^++\dim_\mathrm{A}^\theta F^+)} 
\end{eqnarray*}
which proves that $\dim_\mathrm{A}^\theta (E\times F)\leq \dim_\mathrm{A}^\theta E+\dim_\mathrm{A}^\theta F$.  On the other hand for any $E'\subset E$:
\begin{eqnarray*}
N(B(x,R)\cap E\times F,R^{1/\theta}) &\geq & N(B(x,R)\cap E'\times F,R^{1/\theta}) \\ 
&\geq & \inf_{y\in E'}N(B(y,R)\cap E',R^{1/\theta})N(B(P_Fx,R)\cap F,R^{1/\theta}) .
\end{eqnarray*}
Since $x\in E'\times F$ can be chosen such that
\[
N(B(P_Fx,R)\cap F,R^{1/\theta}) \ \geq  \ \sup_{z\in F}N(B(z,R)\cap F,R^{1/\theta})^-
\]
we have
\begin{eqnarray*}
\sup_{x\in E\times F}N(B(x,R),R^{1/\theta}) &\geq & \sup_{x\in E'\times F}N(B(x,R)\cap E'\times F,R^{1/\theta}) \\ 
&\geq & \inf_{y\in E'}N(B(y,R)\cap E',R^{1/\theta})\sup_{z\in F}N(B(z,R)\cap F,R^{1/\theta})^- .
\end{eqnarray*}
This implies that:
\begin{eqnarray*}
\sup_{x\in E\times F}N(B(x,R),R^{1/\theta}) &\geq & \inf_{y\in E'}N(B(y,R)\cap E',R^{1/\theta})\sup_{z\in F}N(B(z,R)\cap F,R^{1/\theta})
\end{eqnarray*}
which, similar to above, yields $\dim_\mathrm{A}^\theta (E\times F)\geq \dim_\mathrm{ML}^\theta E+\dim_\mathrm{A}^\theta F$ as required.  The second chain of inequalities (which concern the lower spectrum) follow by a similar argument, which we omit.   The third chain of inequalities (which concern the modified lower spectrum) follow easily from the second.  In particular, for the lower bound choose nonempty subsets $E' \subset E$ and $F' \subset F$ such that $\dim_\mathrm{L}^\theta E' \geq \dim_\mathrm{ML}^\theta E^-$ and $\dim_\mathrm{L}^\theta F' \geq \dim_\mathrm{ML}^\theta F ^-$ and then apply monotonicity and the result for the lower spectrum to obtain
\begin{eqnarray*}
\dim_\mathrm{ML}^{\theta}(E \times F)\geq  \dim_\mathrm{ML}^{\theta}(E' \times F')\geq  \dim_\mathrm{L}^{\theta}(E' \times F')&\geq& \dim_\mathrm{L}^{\theta}E'+\dim_\mathrm{L}^{\theta}F' \\ 
& \geq& \dim_\mathrm{ML}^{\theta}E^-+\dim_\mathrm{ML}^{\theta}F^-
\end{eqnarray*}
which proves the desired lower bound.  For the upper bound, the upper bound concerning the lower spectrum implies that
\[
\sup_{E' \subseteq E} \dim_\mathrm{L}^{\theta}(E' \times F) \leq \sup_{E' \subseteq E} \dim_\mathrm{L}^{\theta}E'+\dim_\mathrm{A}^{\theta}F = \dim_\mathrm{ML}^{\theta}E+\dim_\mathrm{A}^{\theta}F
\]
which is almost what we want, apart from that it is not \emph{a priori} obvious that the quantity on the left is equal to $\dim_\mathrm{ML}^{\theta}(E \times F)$.  However, this follows since for any $K \subseteq E \times F$ we have for any $x \in K$ and $R>0$ that
\[
B(x,R) \cap K \subseteq B(x,R) \cap P_EK \times F
\]
which yields that $\dim_\mathrm{L}^{\theta} K \leq \dim_\mathrm{L}^{\theta} (P_EK \times F )$ completing the proof.
\end{proof}

We also obtain a sharp result for `self-products'.
\begin{prop}[Self-products]
Let $F$ be a metric space, $n \in \mathbb{N}$, and equip the $n$-fold product $F^n = F \times \cdots \times  F$ with any suitable product metric.  For any $\theta \in (0,1)$ we have
\[
\dim_\mathrm{A}^{\theta} \left( F^n \right) =  n \dim_\mathrm{A}^{\theta}F,
\]
\[
\dim_\mathrm{L}^{\theta}\left( F^n \right)  =  n \dim_\mathrm{L}^{\theta}F
\]
and
\[
\dim_\mathrm{ML}^{\theta} \left( F^n \right)=  n \dim_\mathrm{ML}^{\theta}F.
\]
\end{prop}

\begin{proof}
This proof is similar to the general case and we omit the details.  The key point is that, for a self-product, one may choose $x \in F$ which witnesses the extremal behaviour at some scale and then consider the point $(x, \dots, x ) \in F^n$.  The projection of this point onto every coordinate  then witnesses extremal behaviour and this passes to $F^n$.
\end{proof}

We note that using a similar approach one may also obtain the following minor, but useful,  improvement on \cite[Theorem 2.1]{Fraser}.  Specifically, we upgrade lower dimension to modified lower dimension which is useful in situations where the lower dimension is small for reasons which do not affect other dimensions, for example when the set $E$ contains an isolated point.
\begin{prop}
For metric spaces $E,F$ we have
\[
 \dim_\mathrm{ML}E + \dim_\mathrm{A}F \leq \dim_\mathrm{A}(E \times F)\leq \dim_\mathrm{A}E+\dim_\mathrm{A}F.
\]
\end{prop}

Another important aspect of a dimension is how it behaves under distortion by maps which are `not too wild'.  Indeed, all of the standard notions of dimension, such as the Hausdorff, box, packing, Assouad and lower dimension, are stable under bi-Lipschitz distortion, for example.  Relaxing bi-Lipschitz to simply Lipschitz or even H\"older, there are elementary bounds which show that under distortion by an $\alpha$-H\"older map the Hausdorff or box dimensions cannot increase by more than a factor of $1/\alpha$, see \cite[Chapter 2-3]{Falconer}.  Assouad and lower dimension do not enjoy such stability and can wildly increase under distortion by even a Lipschitz map, see \cite{Fraser}.  The reason for this is that because one is trying to control two scales (one in each direction), one needs bounds on the distortion of the map in \emph{both} directions.  Here we conduct a detailed analysis of how the dimension spectra distorts under bi-H\"older maps, i.e., H\"older maps with H\"older inverses.  It is noteworthy that one cannot relate the value of the  Assouad spectrum of the set and its image at a particular value $\theta$, but rather  at two different values of $\theta$ which are related according to the H\"older parameters.  This makes the theory of dimension distortion  for our  spectra rather more subtle than for a dimension which returns a single exponent.  Recall that a doubling metric space is one for which there is a uniform constant $C$ such that any ball may be covered by fewer than $C$ balls of half the radius. This is easily seen to be equivalent to having finite Assouad dimension, see \cite[Lemma 9.4]{Robinson}.

\begin{prop}[H\"older maps] \label{Holderthm}
Let $S : X\rightarrow Y$ be a map between doubling metric spaces  $(X,d_X)$ and $(Y,d_Y)$ such that for all $x,y \in X$ with  $d_X(x,y)$ sufficiently  small,
$$d_X(x,y)^\beta \ \lesssim \ d_Y\big(S(x),S(y)\big) \ \lesssim \  d_X(x,y)^\alpha$$ for some fixed constants $\beta\geq 1\geq\alpha>0$.  Then, for any $F\subseteq X$ and $\theta \in (0,1)$, we have
$$\frac{1-\frac{\beta}{\alpha}\theta}{\beta(1-\theta)} \,  \dim_\mathrm{A}^{\frac{\beta}{\alpha}\theta} F \  \leq \  \dim_\mathrm{A}^{\theta} S(F) \  \leq \  \frac{1-\frac{\alpha}{\beta}\theta}{\alpha(1-\theta)} \, \dim_\mathrm{A}^{\frac{\alpha}{\beta}\theta} F$$ 
$$\frac{1-\frac{\beta}{\alpha}\theta}{\beta(1-\theta)} \, \dim_\mathrm{L}^{\frac{\beta}{\alpha}\theta} F \  \leq \   \dim_\mathrm{L}^{\theta} S(F) \  \leq \  \frac{1-\frac{\alpha}{\beta}\theta}{\alpha(1-\theta)} \,\dim_\mathrm{L}^{\frac{\alpha}{\beta}\theta} F$$ 
and
$$\frac{1-\frac{\beta}{\alpha}\theta}{\beta(1-\theta)} \, \dim_\mathrm{ML}^{\frac{\beta}{\alpha}\theta} F \  \leq \   \dim_\mathrm{ML}^{\theta} S(F) \  \leq \  \frac{1-\frac{\alpha}{\beta}\theta}{\alpha(1-\theta)} \,\dim_\mathrm{ML}^{\frac{\alpha}{\beta}\theta} F.$$
\end{prop}

\begin{proof}
First notice that $S$ is invertible and for all $x,y \in S(X)$ with $d_Y(x,y)$ sufficiently  small we have
$$d_Y(x,y)^{1/\alpha} \ \lesssim  \   d_X\big( S^{-1}(x), S^{-1}(y) \big) \  \lesssim \ d_Y(x,y)^{1/\beta}.$$
By the assumptions on $S$ there are uniform constants $C,c>0$ such that for any sufficiently small $r>0$ and $x \in X$ and $y \in S(X)$ we have
\[
B(S(x),cr^{\beta})\subseteq S(B(x,r))\subseteq B(S(x), Cr^{\alpha})
\]
and
\[
B(S^{-1}(y),cr^{1/\alpha})\subseteq S^{-1}(B(y,r))\subseteq B(S^{-1}(y), Cr^{1/\beta}).
\]
Therefore, for $0<r<R$ with $R$ small enough (recall that our metric space has the doubling property) we have
$$N(B(x,R^{1/\alpha}),r^{1/\beta}) \lesssim N(B(S(x),R),r)\lesssim N(B(x,R^{1/\beta}),r^{1/\alpha})$$
 for any $x \in X$.  The left inequality holds because any $r$-cover of $B(S(x),R)$ can be mapped under $S^{-1}$ to yield an (up to multiplicative constants) $r^{1/\beta}$-cover of $B(x,R^{1/\alpha})$ by the same number of sets (up to another multiplicative constant depending on the doubling property of the space).  Similarly, the right inequality holds because any $r^{1/\alpha}$-cover of $B(x,R^{1/\beta})$ can be mapped under $S$ to yield an (up to multiplicative constants)  $r$-cover of $B(S(x),R)$ by the same number of sets (up to another multiplicative constant depending on the doubling property of the space).

Setting $r=R^{1/\theta}$ from here we notice that for any sufficiently small $R>0$ we have by definition
\[
N(B(x,R^{1/\beta}),r^{1/\alpha}) \ \lesssim \   \left( \frac{R^{1/\beta}}{(R^{1/\beta})^{\frac{\beta}{\alpha \theta}}}\right)^{\dim_\mathrm{A}^{\frac{\alpha\theta}{\beta}}F^+}  \ = \  (R^{1-1/\theta})^{\frac{1-\frac{\alpha}{\beta}\theta}{\alpha(1-\theta)}\dim_\mathrm{A}^{\frac{\alpha\theta}{\beta}}F^+}
\]
and, similarly, for infinitely many $R \to 0$ we have
\[
N(B(x,R^{1/\alpha}),r^{1/\beta}) \ \gtrsim \   \left( \frac{R^{1/\alpha}}{(R^{1/\alpha})^{\frac{\alpha}{\beta \theta}}}\right)^{\dim_\mathrm{A}^{\frac{\beta\theta}{\alpha}}F^-}  \ = \  (R^{1-1/\theta})^{\frac{1-\frac{\beta}{\alpha}\theta}{\beta(1-\theta)}\dim_\mathrm{A}^{\frac{\beta\theta}{\alpha}}F^-}.
\]
Recall that if $\beta \theta/\alpha \geq 1$, then $\dim_\mathrm{A}^{\frac{\beta\theta}{\alpha}}F = 0$.  Also by definition, for any sufficiently small $R>0$ we have
$$N(B(S(x),R),r)\ \lesssim \ (R^{1-1/\theta})^{\dim_\mathrm{A}^{\theta} S(F)^+}$$
and for infinitely many $R \to 0$ we have
$$N(B(S(x),R),r)\ \gtrsim \ (R^{1-1/\theta})^{\dim_\mathrm{A}^{\theta} S(F)^-}.$$
Together these estimates yield that for infinitely many $R \to 0$ we have
\[
(R^{1-1/\theta})^{\dim_\mathrm{A}^{\theta} S(F)^-} \lesssim (R^{1-1/\theta})^{\frac{1-\frac{\alpha}{\beta}\theta}{\alpha(1-\theta)}\dim_\mathrm{A}^{\frac{\alpha\theta}{\beta}}F^+}
\]
and
\[
(R^{1-1/\theta})^{\frac{1-\frac{\beta}{\alpha}\theta}{\beta(1-\theta)}\dim_\mathrm{A}^{\frac{\beta\theta}{\alpha}}F^-} \ \lesssim\  (R^{1-1/\theta})^{\dim_\mathrm{A}^{\theta} S(F)^+}   
\]
which gives
$$\frac{1-\frac{\beta}{\alpha}\theta}{\beta(1-\theta)}\dim_\mathrm{A}^{\frac{\beta\theta}{\alpha}}F \ \leq \  \dim_\mathrm{A}^{\theta}S(F) \ \leq \  \frac{1-\frac{\alpha}{\beta}\theta}{\alpha(1-\theta)}\dim_\mathrm{A}^{\frac{\alpha\theta}{\beta}}F$$
as required. The argument for the lower spectrum is similar and omitted.  Since subsets of $F$ are in one to one correspondence with subsets of $S(F)$ through the map $S$, we may take supremum over nonempty subsets of $F$ throughout, which yields the analogous estimates for the modified lower spectrum, completing the proof.
\end{proof}

The lower bounds for the  spectra of $S(F)$ all become equal to 0 (and thus trivial) when $\theta \geq \alpha/\beta$ and the upper bounds for the spectra of $S(F)$ blow up as $\theta \to 1$.  These are unfortunate properties but are indicative of the complex relations between the spectra at different values of $\theta$ once the set has been distorted by $S$.  One can rectify this situation somewhat by combining our estimates with Lemmas \ref{BB} and \ref{BBL} and the classical results that upper and lower box dimension cannot increase by more than a factor of $1/\alpha$ under distortion by an  $\alpha$-H\"older map. 

\begin{cor} \label{Holdercor}
Let $S : X\rightarrow Y$ be as in Proposition \ref{Holderthm}.  Then for any $F \subseteq X$ and $\theta \in (0,1)$
\[
\frac{1-\frac{\beta}{\alpha}\theta}{\beta(1-\theta)}\dim_\mathrm{A}^{\frac{\beta\theta}{\alpha}}F  \, \vee \,   \frac{\overline{\dim}_\mathrm{B} F}{\beta}    \   \ \leq \  \dim_\mathrm{A}^{\theta} S(F) \  \leq \  \frac{1-\frac{\alpha}{\beta}\theta}{\alpha(1-\theta)}\dim_\mathrm{A}^{\frac{\alpha}{\beta}\theta} F 
\]
\[ \frac{1-\frac{\beta}{\alpha}\theta}{\beta(1-\theta)}\dim_\mathrm{L}^{\frac{\beta\theta}{\alpha}}F \ \leq \ \dim_\mathrm{L}^{\theta} S(F) \  \leq \  \frac{1-\frac{\alpha}{\beta}\theta}{\alpha(1-\theta)}\dim_\mathrm{L}^{\frac{\alpha}{\beta}\theta} F  \, \wedge \,  \frac{\underline{\dim}_\mathrm{B} F}{\alpha}
\]
and
\[ \frac{1-\frac{\beta}{\alpha}\theta}{\beta(1-\theta)}\dim_\mathrm{ML}^{\frac{\beta\theta}{\alpha}}F \ \leq \ \dim_\mathrm{ML}^{\theta} S(F) \  \leq \  \frac{1-\frac{\alpha}{\beta}\theta}{\alpha(1-\theta)}\dim_\mathrm{ML}^{\frac{\alpha}{\beta}\theta} F  \, \wedge \,  \frac{\underline{\dim}_\mathrm{B} F}{\alpha}.
\]
\end{cor}

Note that in the above we could also  bound $ \dim_\mathrm{A}^{\theta} S(F) $ from above by  $\dim_\mathrm{A} S(F) $, but the point is to bound dimensions of $S(F)$ by expressions involving only dimensions of $F$ and the bi-H\"older restrictions on $S$ are not enough to yield bounds for the \emph{Assouad dimension} of $S(F)$ in terms of $F$.

Also, we could have used Lemma \ref{BB} to apparently improve the upper bound for the Assouad spectrum to include the bound
\[
\dim_\mathrm{A}^{\theta} S(F) \  \leq \   \frac{\overline{\dim}_\mathrm{B} S(F)}{(1-\theta)}  \  \leq \   \frac{\overline{\dim}_\mathrm{B} F}{\alpha(1-\theta)}
\]
but by virtue of  Lemma \ref{BB} it follows that
\[
\overline{\dim}_\mathrm{B} F \ \geq \   \left(1-\frac{\alpha}{\beta}\theta\right)\dim_\mathrm{A}^{\frac{\alpha}{\beta}\theta} F
\]
for all $\theta \in (0,1)$ and so this estimate cannot improve the one we already have.

It is important to comment on the sharpness of the estimates from Corollary \ref{Holdercor}.  A first observation is that such estimates cannot possibly be sharp in any precise sense, although letting $\alpha,\beta \to 1$ shows that they  are at least asymptotically sharp.  The reason for this is that they are based on knowledge of the extremal distortion of $F$ over the whole space and the spectra are only sensitive to the extremal properties of the set in question.  Indeed, the thickest part of the set $F$ (and $S(F)$), which determines the spectra,  may occur at a location in the domain of $S$ where the distortion is less than the global extreme.   We will consider our estimates in detail for a natural family of sets and bi-H\"older maps in Section \ref{seq1}.

Setting $\alpha=\beta=1$ in Proposition \ref{Holderthm}, we obtain bi-Lipschitz stability as another immediate corollary.
\begin{cor} \label{bilipcor}
The Assouad, lower, and modified lower, spectra are bi-Lipschitz invariant.
\end{cor}

Being bi-Lipschitz invariant is a useful property and one possible application is in classifying metric spaces up to bi-Lipschitz equivalence.  There has been considerable interest in this problem since the seminal paper of Falconer and Marsh \cite{FalconerMarsh} which sought to determine for which pairs of self-similar subsets of the line one can find a bi-Lipschitz function taking one to the other.  Having the same Hausdorff dimension is a necessary condition, since Hausdorff dimension is a bi-Lipschitz invariant, but it is not sufficient: there are self-similar subsets of the line which have the same Hausdorff dimension but which are not bi-Lipschitz equivalent.  As such, it is useful to find other bi-Lipschitz invariants, such as the other notions of dimension mentioned above.  Corollary \ref{bilipcor} provides a new continuum of bi-Lipschitz invariants and so has potential applications in proving that certain metric spaces are not bi-Lipschitz equivalent, even if their Hausdorff, box, packing, Assouad and lower dimensions are equal.

Suppose $S$ is a map on $F$ such that
\[
\frac{\log |x-y|}{\log |S(x) - S(y) |} \to 1
\]
uniformly as $|x-y| \to 0$.  Such maps are sometimes called \emph{quasi-Lipschitz}, see \cite{quasiassouad, quasiassouad2}.  Rather than setting $\alpha=\beta=1$ in Proposition \ref{Holderthm}, if we just let $\alpha, \beta \to 1$ we see that the spectra are also all invariant under \emph{quasi-Lipschitz} maps.

\begin{cor} \label{quasibilipcor}
The Assouad, lower, and modified lower, spectra are quasi-Lipschitz invariant.
\end{cor}

\subsection{Bi-H\"older distortion for Assouad dimension}

In Proposition \ref{Holderthm} we gave some estimates for the Assouad spectrum of a set after distortion  by a  bi-H\"older map. Similar, but simpler, estimates hold for the other standard notions of dimension, such as the Hausdorff dimension and upper and lower box dimension.  In light of the other known results, one might expect that:
\[
\frac{1}{\beta}\dim_\mathrm{A}F    \   \ \leq \  \dim_\mathrm{A} S(F) \  \leq \  \frac{1}{\alpha}\dim_\mathrm{A}F
\]
for any bi-H\"older map with parameters $0<\alpha \leq 1 \leq \beta< \infty$.  In particular, these bounds hold if Assouad dimension is replaced by Hausdorff, packing, upper box or lower box dimension.  The situation turns out to be more subtle for Assouad dimension.   In particular,  the Assouad dimension may be distorted by an absolute constant for bi-H\"older maps with parameters arbitrarily close to 1, i.e., maps which are arbitrarily close to being bi-Lipschitz.  In fact L\"u and Xi \cite[Proposition 1.2]{quasiassouad2} proved that for any $s,t \in (0,1]$ one may find subsets of $[0,1]$ with Assouad dimension $s$ and $t$ respectively, such that one is a quasi-Lipschitz image of the other and vice versa. This shows that there do not exist general bounds on the Assouad dimension of $S(F)$ in terms of the Assouad dimension of $F$ and the H\"older parameters $\alpha$ and $\beta$.  However, we prove that if one assumes that the Assouad spectrum reaches the Assouad dimension, then one \emph{can} give non-trivial dimension bounds for distortion under bi-H\"older maps.  In particular, to obtain some bounds one needs additional assumptions about the set $F$.

\begin{thm} \label{Holderassouadcor}
Let $S : X\rightarrow Y$ be as in Proposition \ref{Holderthm} and let
\[
\theta_0=\inf\left\{ \theta\in [0,1] : \dim_\mathrm{A}^{\theta} F=\dim_\mathrm{A} F \right\}
\]
assuming the set of suitable $\theta$s is non-empty. Then for any $F \subseteq X$ we have
\[
\dim_\mathrm{A} S(F)\geq \frac{\dim_\mathrm{A} F}{\beta-\theta_0\alpha}(1-\theta_0).
\]
In particular, if $S$ is quasi-Lipschitz and $\dim_\mathrm{A}^{\theta} F=\dim_\mathrm{A} F$ for some $\theta$, then $\dim_\mathrm{A} S(F)\geq\dim_\mathrm{A} F$.
\end{thm}

Before we prove the result, note that if $\theta_0$ exists and $\dim_\mathrm{A} F > \overline{\dim}_\mathrm{B} F$, then Proposition \ref{BB} implies that
\[
\theta_0 \geq  1- \frac{\overline{\dim}_\mathrm{B} F}{\dim_\mathrm{A} F} > 0.
\]
\begin{proof}
We have for any $\theta\in [0,1]$ that
$$\dim_\mathrm{A} S(F)\geq \dim_\mathrm{A}^\theta S(F)$$
and Proposition \ref{Holderthm} further implies that
$$\dim_\mathrm{A}^\theta S(F)\geq \frac{1-\frac{\beta}{\alpha}\theta}{\beta(1-\theta)} \dim_\mathrm{A}^{\frac{\beta}{\alpha}\theta} F.  $$
Therefore, applying these inequalities with $\theta=\frac{\alpha}{\beta}\theta_0 \in (0,1)$ proves the desired lower bound.
\end{proof}

Observe that the smaller $\theta_m$ is, the better our lower bound for the Assouad dimension of $S(F)$.   It is natural to consider an analogous upper bound, but this would require \emph{a priori} knowledge of the set $S(F)$, i.e., we have to know $\theta_0'=\inf\lbrace \theta\in [0,1] : \dim_\mathrm{A}^{\theta} S(F)=\dim_\mathrm{A} S(F) \rbrace $.  One may obtain a precise analogue by applying the above theorem with $S$ replaced by $S^{-1}$, but we do not pursue the details here.

We can also derive similar results for the lower and modified lower dimension using Proposition \ref{Holderthm}, but we leave the precise formulations to the reader.

\section{Measureability properties}  \label{BorelSect}

In this section we consider the Borel measureability of the Assouad and lower spectra.  This question has  previously been investigated in related contexts.  For example, Mattila and Mauldin \cite{mattilamauldin} proved that the Hausdorff dimension and upper and lower box dimensions are Borel measurable and, moreover, are of Baire class 2, but packing dimension is \emph{not} Borel measurable.  In \cite{Fraser} it was shown that the Assouad and lower dimensions are Borel measureable.  More precisely, it was shown that Assouad dimension is Baire 2, but it was only shown that lower dimension is Baire 3, leaving open the possibility that it is Baire 2 \cite[Question 4.6]{Fraser}.  We solve this problem here by proving that the lower dimension is in fact \emph{precisely} Baire 2.

The \emph{Baire hierarchy} is used to classify functions by their `level of discontinuity' and is formulated for functions between metric spaces $(A, d_A)$ and $(B, d_B)$ as follows.  A function $f: A \to B$ is  Baire  0 if it is continuous.  The latter classes are defined inductively by saying that a function $f:A \to B$ is Baire  $n+1$ if it is a pointwise limit of a sequence of Baire  $n$ functions.  Thus functions lying in higher Baire classes (and not in lower ones) are set theoretically further away  from being continuous.  For more details on the Baire hierarchy, see \cite{kechris}. In particular, if a function belongs to any Baire class, then it is Borel measurable.  

Let $\mathcal{K}(X)$ denote the set of all non-empty compact subsets of a non-empty compact metric space $X$ and endow this space with the Hausdorff metric, $d_\mathcal{H}$, defined by
\[
d_\mathcal{H}(E,F) = \inf \{ \varepsilon>0: E \subseteq F_\varepsilon \text{ and } F \subseteq E_\varepsilon\}
\]
for $E,F \in \mathcal{K}(X)$ and where $E_\varepsilon$ denotes the $\varepsilon$-neighbourhood of $E$.  The metric space $(\mathcal{K}(X),d_\mathcal{H})$ is compact.  Equip the product space  $\mathcal{K}(X) \times (0,1)$ with the product topology and any compatible metric. 
\begin{thm} \label{meas1}
The function $\Delta_\text{\emph{A}}: \mathcal{K}(X) \times (0,1) \to \mathbb{R}$ defined by
\[
\Delta_\text{\emph{A}}(F,\theta) = \dim_\text{\emph{A}}^\theta F
\]
is of Baire class 2 and, in particular, Borel measurable. Moreover, it is not in general of Baire class 1.
\end{thm}

\begin{thm} \label{meas2}
The function $\Delta_\text{\emph{L}}: \mathcal{K}(X) \times (0,1) \to \mathbb{R}$ defined by
\[
\Delta_\text{\emph{L}}(F, \theta) = \dim_\text{\emph{L}}^\theta F
\]
is of Baire class 2 and, in particular, Borel measurable. Moreover, it is not in general of Baire class 1.
\end{thm}

We also answer \cite[Question 4.6]{Fraser} in the affirmative:

\begin{thm} \label{meas3}
The function $\Delta'_\text{\emph{L}}: \mathcal{K}(X) \to \mathbb{R}$ defined by
\[
\Delta'_\text{\emph{L}}(F) = \dim_\text{\emph{L}} F
\]
is of Baire class 2.
\end{thm}

We will prove Theorems \ref{meas1},  \ref{meas2} and  \ref{meas3} in Section \ref{BorelProof}.  Finally we remark that, by using similar techniques, one may also prove that for a fixed $\theta \in (0,1)$ the maps $F \mapsto \dim_\mathrm{A}^\theta F$ and $F \mapsto \dim_\mathrm{L}^\theta F$ are  Baire 2 (and not Baire 1), but we omit the details. 

Recall that for a fixed $F \in \mathcal{K}(X)$ the maps $\theta  \mapsto \dim_\mathrm{A}^\theta F$ and $\theta \mapsto \dim_\mathrm{L}^\theta F$ are continuous (Baire 0), see Corollary \ref{Con} and Theorem \ref{Con2}.  We are unaware of any \emph{a priori} way of relating Baire classes of a function on a product space which is separately measureable to the Baire classes of its fibre maps (in this case 2 and 0). A classical result of Lebesgue shows that if a function on a product space is separately continuous, then it is no worse than Baire 1, see \cite{Lebesgue, Rudin}.

\subsection{Semicontinuity of covering and packing functions}

We first establish some technical semicontinuity properties, from which the desired Baire classifications will swiftly follow in the subsequent section.  For clarity we write $\overline{B}(x,R)$ for the \emph{closed} ball centered at $x \in X$ with radius $R>0$ and $B^0(x,R)$ for the \emph{open} ball centered at $x \in X$ with radius $R>0$.  Also, $N(F,r)$ will denote the smallest number of \emph{open} sets required for an $r$-cover of $F \subseteq X$ and $M(F,r)$ will denote the maximum number of \emph{closed} balls in an $r$-packing of $F\subseteq X$, where an $r$-packing of $F$ is a pairwise disjoint collection of closed balls centered in $F$ of radius $r$.  For convenience we assume that $N(\emptyset, r) = M(\emptyset, r) = 0$ for any $r>0$.

We will establish semicontinuity of six related covering or packing functions, which we define now.  For fixed $R \in (0,1)$ define four functions $\mathcal{N}_R^{\sup}$, $\mathcal{N}_R^{\inf}$, $\mathcal{M}_R^{\sup}$, $\mathcal{M}_R^{\sup}: \mathcal{K}(X) \times (0,1) \to \mathbb{R}$ by
\[
\mathcal{N}_R^{\sup}(F, \theta) = \sup_{x \in F} N( \overline{B}(x,R) \cap F, \, R^{1/\theta} )
\]
\[
\mathcal{N}_R^{\inf}(F, \theta) = \inf_{x \in F} N( \overline{B}(x,R) \cap F, \, R^{1/\theta} )
\]
\[
\mathcal{M}_R^{\sup}(F, \theta) = \sup_{x \in F} M( B^0(x,R) \cap F, \, R^{1/\theta} )
\]
and
\[
\mathcal{M}_R^{\inf}(F, \theta) = \inf_{x \in F} M( B^0(x,R) \cap F, \, R^{1/\theta} ).
\]
Also, for fixed $R \in (0,1)$ and $r \in (0,R)$ define $\mathcal{N}_{r,R}^{\inf}$, $\mathcal{M}_{r,R}^{\inf} : \mathcal{K}(X)  \to \mathbb{R}$ by
\[
\mathcal{N}_{r,R}^{\inf}(F) = \inf_{x \in F} N( \overline{B}(x,R) \cap F, \, r )
\]
and
\[
\mathcal{M}_{r,R}^{\inf}(F) = \inf_{x \in F} M( B^0(x,R) \cap F, \, r ).
\]
It is useful to note that all of the infimums and supremums in the above definitions are actually minimums and maximums, respectively.  This is because the packing and covering numbers only take positive integer values.

\begin{lma} \label{semicontinuity}
For fixed $R \in (0,1)$ and $r \in (0,R)$ the following semicontinuity properties hold:
\begin{enumerate}
\item $\mathcal{N}_R^{\sup}$ is upper semicontinuous
\item $\mathcal{N}_R^{\inf}$ is upper semicontinuous
\item $\mathcal{M}_R^{\sup}$ is lower semicontinuous
\item $\mathcal{M}_R^{\inf}$ is lower semicontinuous
\item $\mathcal{N}_{r,R}^{\inf}$ is upper semicontinuous
\item $\mathcal{M}_{r,R}^{\inf}$ is lower semicontinuous
\end{enumerate}
\end{lma}

\begin{proof}
Fix $R \in (0,1)$ and $r \in (0,R)$.
\begin{enumerate}
\item  To prove that $\mathcal{N}_R^{\sup}$ is upper semicontinuous, it suffices to show that for any $t \in \mathbb{R}$, the set
\begin{equation} \label{closedset1}
\left\{ (F, \theta) \in \mathcal{K}(X) \times (0,1) : \mathcal{N}_R^{\sup}(F, \theta) \geq t\right\}
\end{equation}
is closed. Fix $t \in \mathbb{R}$ and let $(F_i, \theta_i)$ be a convergent sequence of pairs from the above set, with limit $(F, \theta) \in \mathcal{K}(X) \times (0,1)$. Thus for each $i$ we may find $x_i \in F_i$ such that
\[
N( \overline{B}(x_i,R) \cap F_i, \, R^{1/\theta_i} ) \geq t.
\]
Take a subsequence of the sets $\overline{B}(x_i,R) \cap F_i$ which converges in the Hausdorff metric, which we may do by compactness.  Denote the Hausdorff limit of this sequence of sets by $Y$ and note that $Y \subseteq  \overline{B}(x,R) \cap F$. Let $\{U_i\}$ be any finite open $R^{1/\theta}$-cover of $\overline{B}(x,R) \cap F$.   We may choose $i$ large enough to guarantee that $\{U_i\}$ is an open $R^{1/\theta}$-cover of $\overline{B}(x_i,R) \cap F_i$ and, moreover, we may simultaneously choose $i$ large enough such that we may find an open $R^{1/\theta_i}$-cover of $\overline{B}(x_i,R) \cap F_i$ by the same number of sets. If $\theta_i\geq\theta$ then $\{U_i\}$ suffices and if $\theta_i<\theta$ then  this can be achieved by shrinking the sets $\{U_i\}$ slightly, using the fact that this is an open cover of a closed set and so there is room to do so.  This guarantees
\[
N( \overline{B}(x,R) \cap F, \, R^{1/\theta} ) \geq t
\]
and thus the limit $(F, \theta)$ belongs to (\ref{closedset1}), proving that it is closed.
\item To prove that  $\mathcal{N}_R^{\inf}$ is upper semicontinuous, it suffices to show that for any $t \in \mathbb{R}$, the set
\begin{equation} \label{openset1}
\left\{ (F, \theta) \in \mathcal{K}(X) \times (0,1) : \mathcal{N}_R^{\inf}(F, \theta) < t\right\}
\end{equation}
is open. Fix $t \in \mathbb{R}$ and let $(F, \theta)$ be a pair from the above set.  Thus we may find $x \in F$ such that
\[
N( \overline{B}(x,R) \cap F, \, R^{1/\theta} ) < t.
\]
Let $\{U_i\}$ be an open $R^{1/\theta}$-cover of $\overline{B}(x,R) \cap F$ by $N( \overline{B}(x,R) \cap F, \, R^{1/\theta} ) $ sets.  If $F'$ is sufficiently close to $F$ in the Hausdorff metric, then we may find $x' \in F'$ such that $\{U_i\}$ is also an open $R^{1/\theta}$-cover of $\overline{B}(x',R) \cap F'$.  Moreover, if $\theta'$ is sufficiently close to $\theta$, then we may distort $\{U_i\}$ to obtain an open $R^{1/\theta'}$-cover of $\overline{B}(x',R) \cap F'$ by the same number of sets.  This guarantees that if $(F',\theta')$ is sufficiently close to $(F, \theta)$ in the product metric, then
\[
\mathcal{N}_R^{\inf}(F', \theta') \leq N( \overline{B}(x,R) \cap F, \, R^{1/\theta} ) < t
\]
which proves that (\ref{openset1}) is open.
\item To prove that $\mathcal{M}_R^{\sup}$ is lower semicontinuous, it suffices to show that for any $t \in \mathbb{R}$, the set
\begin{equation} \label{openset2}
\left\{ (F, \theta) \in \mathcal{K}(X) \times (0,1) : \mathcal{M}_R^{\sup}(F, \theta) > t\right\}
\end{equation}
is open. Fix $t \in \mathbb{R}$ and let $(F, \theta)$ be a  pair from the above set.  Thus we may find $x \in F$ such that
\[
M( B^0(x,R) \cap F, \, R^{1/\theta} ) > t.
\]
Let $\{C_i\}$ be an $R^{1/\theta}$-packing of $B^0(x,R) \cap F$ by $M( B^0(x,R) \cap F, \, R^{1/\theta} ) $  closed balls.  If $F'$ is sufficiently close to $F$ in the Hausdorff metric, then we may find $x' \in F'$ close to $x$ and $\{C'_i\}$ with $C_i'$ close to $C_i$ such that $\{C'_i\}$ is an $R^{1/\theta}$-packing of $B^0(x',R) \cap F'$ by the same number of closed balls.  Moreover, if $\theta'$ is sufficiently close to $\theta$, then we may distort the $\{C_i'\}$ to obtain an $R^{1/\theta'}$-packing of $B^0(x',R) \cap F'$.  This guarantees that if $(F',\theta')$ is sufficiently close to $(F, \theta)$ in the product metric, then
\[
\mathcal{M}_R^{\sup}(F', \theta') \geq  M( B^0(x,R) \cap F, \, R^{1/\theta} ) > t
\]
which proves that (\ref{openset2}) is open.
\item  To prove that $\mathcal{M}_R^{\inf}$ is lower semicontinuous, it suffices to show that for any $t \in \mathbb{R}$, the set
\begin{equation} \label{closedset2}
\left\{ (F, \theta) \in \mathcal{K}(X) \times (0,1) : \mathcal{M}_R^{\inf}(F, \theta) \leq t\right\}
\end{equation}
is closed. Fix $t \in \mathbb{R}$ and let $(F_i, \theta_i)$ be a convergent sequence of pairs from the above set, with limit $(F, \theta) \in \mathcal{K}(X) \times (0,1)$. Thus for each $i$ we may find $x_i \in F_i$ such that
\begin{equation} \label{contradictionbound}
M(B^0(x_i,R) \cap F_i, \, R^{1/\theta_i} ) \leq t.
\end{equation}
Take a convergent subsequence from the sequence $\{x_i\}$, which we may do by compactness.  Denote the limit of this sequence by $x$ and note that $x \in F$.  We claim that
\[
M( B^0(x,R) \cap F, \, R^{1/\theta} ) \leq t
\]
which shows that $(F, \theta)$ belongs to (\ref{closedset2}), proving that it is closed.  It remains to prove the claim, which we do by contradiction.  Assume that $\{C_i\}$ is an $R^{1/\theta}$-packing of $B^0(x,R) \cap F$ by strictly greater than $t$  closed balls.  Following (3) above, if $(F',\theta')$ is sufficiently close to $(F, \theta)$ in the product metric, then we may find $x' \in F'$ close to $x$ and $\{C'_i\}$ with $C_i'$ close to $C_i$ such that $\{C'_i\}$ is an $R^{1/\theta'}$-packing of $B^0(x',R) \cap F'$ by the same number of closed balls. This contradicts (\ref{contradictionbound}) and thus completes the proof.  We note that it was crucial to our argument that $B^0(x,R)$ was an \emph{open} ball.
\item Upper semicontinuity of $\mathcal{N}_{r,R}^{\inf}$ was already  established in \cite[Lemma 5.6]{Fraser}, where it was written as $\Phi_{r,R}$.  The proof is very similar to (in fact slightly simpler than) the proof of (2) above and is omitted.
\item It was hinted at in \cite{Fraser} that $\mathcal{M}_{r,R}^{\inf}$ may not be lower semicontinuous, but we show here that it is.  This is the key to lowering the Baire class of lower dimension from 3 to the sharp value of 2. The proof is very similar to (in fact slightly simpler than) the proof of (4) above and is omitted.
\end{enumerate}  \vspace{-8mm}
\end{proof}

\subsection{Proofs of measureability results}  \label{BorelProof}

To prove that a function from a metric space $A$ to $\mathbb{R}$ is Baire 2, it suffices to show that any open interval in $\mathbb{R}$  pulls back to a $\mathcal{G}_{\delta \sigma}$ set in $A$, see \cite[Theorem 24.3]{kechris}.  Recall that a set is $\mathcal{G}_{\delta \sigma}$ if it can be expressed as a countable union of sets which are themselves expressible as countable intersections of open sets.  As such, throughout this section let $a,b \in \mathbb{R}$ with $a<b$ and we will consider various pullbacks of the interval $(a,b)$.  We write $\mathbb{Q}^+$ to denote the strictly positive rationals.

We begin by considering the Assouad spectrum.   Straight from the definition, we obtain
\begin{eqnarray*}
\Delta_\text{A}^{-1} \big( (a,b) \big) &=& \left\{ (F, \theta) \in \mathcal{K}(X) \times (0,1) :  \dim_\text{A}^\theta F > a \right\}  \\ \\
 &\,& \qquad   \bigcap \  \left\{(F, \theta) \in \mathcal{K}(X) \times (0,1) :  \dim_\text{A}^\theta  F < b \right\} \\ \\
&=&  \Bigg\{(F, \theta) \in \mathcal{K}(X) \times (0,1)  :  (\exists n \in \mathbb{N}) ( \forall C>0)( \forall \rho>0) (\exists  0<R <\rho ) \\ \\
&\,&  \qquad \qquad \qquad  \sup_{x \in F}   M\Big( B^0\big( x, R \big) \cap F , R^{1/\theta} \Big)  \ > \ C \, \bigg( \frac{R}{R^{1/\theta}} \bigg)^{a+1/n} \Bigg\} \\ \\  
 &\,&    \bigcap \  \Bigg\{(F, \theta) \in \mathcal{K}(X) \times (0,1) : (\exists n \in \mathbb{N}) (\exists C>0) (\exists \rho>0)  (\forall  0<R <\rho ) \\ \\
&\,& \qquad \qquad \qquad  \sup_{x \in F} N \Big( \overline{B}\big( x, R \big) \cap F, R^{1/\theta} \Big)  \ < \ C \, \bigg( \frac{R}{R^{1/\theta}} \bigg)^{b-1/n} \Bigg\}  \\ \\
&=& \left(  \bigcup_{n \in \mathbb{N}} \  \bigcap_{C \in \mathbb{Q}^+} \  \bigcap_{\rho \in \mathbb{Q}^+} \  \bigcup_{R \in \mathbb{Q} \cap (0,\rho)} \ \Big(\mathcal{M}^{\sup}_R\Big)^{-1}\,  \Big( \big(  C \, R^{(1-1/\theta)(a+1/n)} , \infty \big) \Big) \right)  \\ \\
&\,&   \bigcap \ \left( \bigcup_{n \in \mathbb{N}} \  \bigcup_{C \in \mathbb{Q}^+} \ \bigcup_{\rho \in \mathbb{Q}^+} \  \bigcap_{R \in \mathbb{Q} \cap (0,\rho)} \  \Big(\mathcal{N}_R^{\sup} \Big)^{-1}\,  \Big( \big( -\infty,  \, C \, R^{(1-1/\theta)(b-1/n)} \big) \Big) \right)
\end{eqnarray*}
which is $\mathcal{G}_{\delta \sigma}$ by the lower semicontinuity of $\mathcal{M}^{\sup}_R$ and the upper semicontinuity of $\mathcal{N}^{\sup}_R$, see Lemma \ref{semicontinuity} parts (3) and (1) respectively.  This completes the proof that $\Delta_A$ is Baire 2.

We now turn to the lower spectrum, which is similar.  By a similar decomposition argument we obtain
\begin{eqnarray*}
 \Delta_\text{L}^{-1} \big( (a,b) \big) &=&\left(  \bigcup_{n \in \mathbb{N}} \  \bigcup_{C \in \mathbb{Q}^+} \  \bigcup_{\rho \in \mathbb{Q}^+} \  \bigcap_{R \in \mathbb{Q} \cap (0,\rho)} \ \Big(\mathcal{M}^{\inf}_R\Big)^{-1}\,  \Big( \big(  C \, R^{(1-1/\theta)(a+1/n)} , \infty \big) \Big) \right)  \\ \\
&\,&    \bigcap \ \left( \bigcup_{n \in \mathbb{N}} \  \bigcap_{C \in \mathbb{Q}^+} \  \bigcap_{\rho \in \mathbb{Q}^+} \  \bigcup_{R \in \mathbb{Q} \cap (0,\rho)} \  \Big(\mathcal{N}_R^{\inf} \Big)^{-1}\,  \Big( \big( -\infty,  \, C \, R^{(1-1/\theta)(b-1/n)} \big) \Big) \right)
\end{eqnarray*}
which is $\mathcal{G}_{\delta \sigma}$ by the lower semicontinuity of $\mathcal{M}^{\inf}_R$ and the upper semicontinuity of $\mathcal{N}^{\inf}_R$, see Lemma \ref{semicontinuity} parts (4) and (2) respectively.  This completes the proof that $\Delta_L$ is Baire 2.

Finally, it is straightforward to see that neither the Assouad spectrum  nor the lower spectrum are typically Baire 1.  This can be seen by recalling that the points of continuity for a Baire 1 function form a dense $\mathcal{G}_\delta$ set, see \cite[Theorem 24.14]{kechris}, but it is evident that the Assouad and lower spectra are discontinuous everywhere (unless $X$ has a particularly simple form).  In particular, every set in $\mathcal{K}(X)$ can be approximated by finite sets (which all have dimension spectra equal to 0) or by their closed $\varepsilon$-neighbourhoods (which typically have dimension spectra larger than 0).  If $X$ is the Euclidean unit interval  $[0,1]$, for example,  then the  closed $\varepsilon$-neighbourhood of any set in $\mathcal{K}(X)$ has lower spectrum equal  to 1. Of course, pathological examples are possible: if $X$ is a single point, for example, then all of the dimension spectra are trivially continuous.

Finally, we prove that the lower \emph{dimension} is Baire 2, thus answering \cite[Question 4.6]{Fraser}.  Another decomposition argument yields
\begin{align*}
 (\Delta'_\text{L})^{-1} \big( (a,b) \big) = & \left(  \bigcup_{n \in \mathbb{N}} \  \bigcup_{C \in \mathbb{Q}^+} \  \bigcup_{\rho \in \mathbb{Q}^+} \  \bigcap_{R \in \mathbb{Q} \cap (0,\rho)} \ \bigcap_{r \in \mathbb{Q} \cap (0,R)} \ \Big(\mathcal{M}^{\inf}_{r,R}\Big)^{-1}\,  \Big( \big(  C \, (R/r)^{a+1/n} , \infty \big) \Big) \right)  \\ \\
  &  \hspace{-6mm} \bigcap \ \left( \bigcup_{n \in \mathbb{N}} \  \bigcap_{C \in \mathbb{Q}^+} \  \bigcap_{\rho \in \mathbb{Q}^+} \  \bigcup_{R \in \mathbb{Q} \cap (0,\rho)} \   \bigcup_{r \in \mathbb{Q} \cap (0,R)} \  \Big(\mathcal{N}_{r,R}^{\inf} \Big)^{-1}\,  \Big( \big( -\infty,  \, C \, (R/r)^{b-1/n} \big) \Big) \right)
\end{align*}
which is $\mathcal{G}_{\delta \sigma}$ by the lower semicontinuity of $\mathcal{M}^{\inf}_{r,R}$ and the upper semicontinuity of $\mathcal{N}^{\inf}_{r,R}$, see Lemma \ref{semicontinuity} parts (6) and (5) respectively.  This completes the proof that $\Delta'_L$ is Baire 2.  It was observed in \cite{Fraser} that $\Delta'_L$ is not Baire 1, again since it is discontinuous everywhere.

\newpage

\section{Decreasing sequences with decreasing gaps}  \label{SequencesSect}

In this section we will consider a simple family of  countable compact fractal subsets of the line which allow explicit calculation of the Assouad spectra.  Observe that any countable compact set has lower spectra and modified lower spectra equal to zero and so we omit discussion of these spectra for the duration of this section.  Despite being relatively simple, these fractal sequences have several useful properties. First they provide us with a continuum of examples where the upper bound from Proposition \ref{BB} is attained.  Secondly, they provide examples where the \emph{lower} bound from Proposition \ref{BB} is attained, see Example \ref{lowersharpe}.  Thus we demonstrate the sharpness of Proposition \ref{BB}.   They also allow us to analyse the bounds for dimension distortion under bi-H\"older maps given in Corollary \ref{Holdercor} in an explicit and representative way, see Section \ref{seq1}.

More precisely, we study  \emph{decreasing sequences with decreasing gaps}, which we formulate as follows.  Let $f$ be a function from $\mathbb{R}^+$ to $[0,1]$ such that $f(x)$ and $g(x):=f(x)-f(x+1)$ are both strictly decreasing functions and they converge
to $0$ as $x\rightarrow\infty$.  We also assume for convenience that  both $f$ and $g$ are smooth.  Our set of interest is then
\[
F \ = \ \{f(n)\}_{n\geq 1}  \ \cup \ \{0\}.
\]
The following result (stated using our notation) was proved by Garc\'ia, Hare and Mendivil \cite[Proposition 4]{Garcia}.  We also obtained this result, using a different proof, in answer to  a question posed to us by Chris Miller (Ohio State University), but we omit our argument and refer the reader to \cite{Garcia}. We also refer the reader to our proof of Theorem \ref{assspiral}, which proves a similar dichotomy in a different setting.

\begin{thm}[Garc\'ia, Hare and Mendivil]\label{seqdichotomy}
Let $F= \{0\}  \cup \{f(n)\}_{n\in\mathbb{N}}$ be a decreasing sequence with decreasing gaps as described above. Then the Assouad dimension of $F$ is either $0$ or $1$. Moreover, the Assouad dimension is 0 if and only if $f(n)$ decays to 0 exponentially.
\end{thm}

Our main result on dimensions of decreasing sequences with decreasing gaps is as follows and shows that the Assouad spectrum only depends on the upper box dimension of the set.   A pleasant ancillary benefit of the explicit formula we obtain is that for all such sets the upper bound in Proposition \ref{BB} is sharp.  Also, we note that the upper box dimensions of such sets are well studied and can be computed effectively.  For example, the upper (and lower) box dimensions can be estimated in terms of the exponential decay rate of the gap lengths $g(n)$.  See the `cut-out sets' discussed in \cite[Chapter 3]{techniques} for more information.

\begin{thm} \label{sequencesmain}
Let $F= \{0\}  \cup \{f(n)\}_{n\in\mathbb{N}}$ be a decreasing sequence with decreasing gaps as described above.  Then for all $\theta \in (0,1)$ we have
$$\dim_\mathrm{A}^{\theta} F \ = \ \frac{\overline{\dim_\mathrm{B}}F}{1-\theta}\wedge 1.$$
\end{thm}

\begin{proof}If either $\overline{\dim_\mathrm{B}}F=0$ or $\overline{\dim_\mathrm{B}}F=1$ then the result follows immediately from Proposition \ref{BB} and therefore we may assume from now on that
$$\overline{\dim_\mathrm{B}}F=B \in (0,1).$$
We start by giving some general bounds.  Let $0<r<R<1$ and consider the number 
\[
\sup_{x\in F}N(B(x,R)\cap F,r).
\]
 For any $r>0$, there is a smallest number $n_r$ such that $g(n_r)<r$. Notice that by definition $n_r=[g^{-1}(r)]=g^{-1}(r)+O(1)$.   If $R\leq f(n_r)$ then we will need approximately $(R/r)$ many $r$-balls to cover $[0,R)\cap F$, and this is already of the largest order possible. Therefore we will focus on  the case where $R>f(n_r)$. The following bound follows from the fact that the sequence has decreasing gaps:
\begin{eqnarray*}
N(B(0,R/2)\cap F,r) \ \leq  \  \sup_{x\in F}N(B(x,R/2)\cap F,r) \ \leq  \ N(B(0,R)\cap F,r)  \ \leq \  2N(B(0,R/2)\cap F,r).
\end{eqnarray*}
If $R\geq f(n_r)$, then we have the key formula:
\begin{equation}
N(B(0,R)\cap F,r) \ \asymp  \  \frac{f\circ g^{-1}(r)}{r}+g^{-1}(r)-f^{-1}(R). \label{KEY}
\end{equation}
This is because for points smaller than $f\circ g^{-1}(r)$ the gaps are smaller than $r$, and therefore we need approximately $\frac{f\circ g^{-1}(r)}{r}$ many $r$-balls to cover them.  Moreover, for points between $f\circ g^{-1}(r)$ and $R$, of which there are approximately $g^{-1}(r)-f^{-1}(R)$ many, we need one $r$-ball for each of them.

Let $\theta \in (0, 1-B)$ and observe that if $R>0$ is sufficiently small, then $g^{-1}(R^{1/\theta})-f^{-1}(R)\geq 0$. It therefore follows from the key formula (\ref{KEY}) that 
$$\sup_{x \in F} N(B(x,R)\cap F,R^{1/\theta}) \   \asymp \  \frac{f\circ g^{-1}(R^{1/\theta})}{R^{1/\theta}}+g^{-1}(R^{1/\theta})-f^{-1}(R).$$
 If we can find infinitely many $R\to 0$ such that
$$g^{-1}(R^{1/\theta})\leq f^{-1}\left({R^{(1-B^-)/\theta}}\right)$$
then
$$\frac{f\circ g^{-1}(R^{1/\theta})}{R^{1/\theta}}\geq {R^{-B^-/\theta}}$$
and  we get $\dim_\mathrm{A}^{\theta} F\geq\frac{B^-}{1-\theta}$ as required.
Therefore assume that for all sufficiently small $R>0$ we have
$$g^{-1}(R^{1/\theta})> f^{-1}\left({R^{(1-B^-)/\theta}}\right).$$
This implies that for $R$ small enough we have
$$f\circ g^{-1}(R)<{R}^{1-B^-}$$
which in turn implies that for $n$ large enough we have
$$f(n)< g(n)^{1-B^-}=(f(n)-f(n+1))^{1-B^-}$$
and
$$g(n) = f(n)-f(n+1)>f(n)^{1/(1-B^-)}.$$
This holds for all large enough $n$ and therefore we can assume it holds for all $n$ without loss of generality. For simplicity, we write $\alpha=\frac{1}{1-B^-}>1$.  We have
\begin{eqnarray*}
f(n+1)^{1-\alpha}-f(n)^{1-\alpha} &=& (f(n)+f(n+1)-f(n))^{1-\alpha}-f(n)^{1-\alpha} \\ 
&=& f(n)^{1-\alpha}\left(\left( 1+\frac{f(n+1)-f(n)}{f(n)}\right)^{1-\alpha}-1\right) \\ 
&\geq&  f(n)^{1-\alpha}\left( (1-\alpha)\frac{f(n+1)-f(n)}{f(n)}\right) \\ 
&=& (\alpha-1)\frac{g(n)}{f(n)^{\alpha}} \\ 
&>& (\alpha-1).
\end{eqnarray*}
Iterating the above inequality yields
\[
f(n)^{1-\alpha}-f(1)^{1-\alpha} > (\alpha-1)n 
\]
and therefore
\[
f(n) <\left( \frac{1}{(\alpha-1)n+f(1)^{1-\alpha}}\right)^{\frac{1}{\alpha-1}}.
\]
This implies that for all large enough $n$ we have
$$f(n)\lesssim \left(\frac{1}{n}\right)^{\frac{1}{\alpha-1}}=\left( \frac{1}{n}\right)^{\frac{1-B^-}{B^-}}$$
and for small enough $x$ we have
\begin{equation} \label{newgoodbound}
f^{-1}(x)\lesssim \left(\frac{1}{x}\right)^\frac{B^-}{1-B^-}.
\end{equation}
Recalling that $B \in (0,1)$ is the upper box dimension, we can find a sequence $r_i\to 0$ such that:
$$\frac{f(n_{r_i})}{r_i}+n_{r_i}\gtrsim N(F, r_i) \gtrsim {r_i}^{-B^-}.$$
If for infinitely many $i$ we have $$\frac{f(n_{r_i})}{r_i}\gtrsim r_i^{-B^-}$$ then the situation is the same as in the beginning of the proof and we get our conclusion.  Otherwise we have infinitely many $i$ such that $$n_{r_i}\gtrsim {r_i}^{-B^-}.$$
It follows from the key formula (\ref{KEY}), that for infinitely many $i$ we have
\begin{eqnarray*}
\sup_{x \in F} N(B(x,R_i) \cap F, r_i)  &\gtrsim& g^{-1}(r_i)-f^{-1}(R_i) \\ 
&=& g^{-1}(r_i)-f^{-1}({r_i}^\theta)\\ 
&\gtrsim& {r_i}^{-B^-}-\left( \frac{1}{r_i^\theta} \right)^{B^-/(1-B^-)} \qquad \text{by (\ref{newgoodbound})} \\ 
&=&r_i^{-B^-}-r_i^{-B^-\theta/(1-B^-)}.
\end{eqnarray*}
Since  $\theta<1-B^-$  we have  $B^->\frac{B^-\theta}{1-B^-}$ and so
\[
\sup_{x \in F} N(B(x,R_i), r_i)  \gtrsim r_i^{-B^-}
\]
for infinitely many $i$.  It follows that for $\theta \in (0,1-B)$ we have $\dim_\mathrm{A}^{\theta} F\geq B^-/ (1-\theta)$ which, combined with  Proposition \ref{BB}, yields the desired result for this range of $\theta$.  Finally, continuity of the spectrum (Corollary \ref{Con}) gives that for $\theta = 1-B$ we have  $\dim_\mathrm{A}^{\theta} F = 1 = \dim_\mathrm{A} F$ and Corollary \ref{stayA} yields that this also holds for all $\theta \in [ 1-B,1)$ completing the proof.
\end{proof}

\begin{example} \label{lowersharpe}
The set  $E = \{e^{-\sqrt{n}} : n \in \mathbb{N}\} \cup \{0\}$ is a simple example where the spectrum does not peak at the Assouad dimension.  Straightforward arguments, which we omit, yield that $\dim_\mathrm{B} E =   \dim_\mathrm{A}^{\theta} E = 0 <  \dim_\mathrm{A} E = 1$.  Moreover, this can example can be modified to provide constructions demonstrating the sharpness of the lower bound from Proposition \ref{BB}, even if the box dimension is positive.  For example, consider $F:=[0,1] \times E$.  It follows from the discussion here and Proposition \ref{basic3} that
\[
\dim_\mathrm{B} F  \ =  \   \dim_\mathrm{A}^{\theta} F \ = \  1 \  < \   \dim_\mathrm{A} F  \ = \  2.
\]
\end{example}

\subsection{Sequences with polynomial decay}  \label{seq1}

In this section provide our first concrete example where we can compute the Assouad spectrum explicitly. We specialise to a particular continuously parameterised family of decreasing sequences with prescribed polynomial decay.  In particular, for a fixed $\lambda>0$ we study the set
\[
F_\lambda \ = \  \{0\} \, \cup \, \left\{ \frac{1}{n^{\lambda}} \right\}_{n \in \mathbb{N}}
\]
and give an explicit formula for $\dim_\mathrm{A}^{\theta} F_\lambda$.  These sets $F_\lambda$ are some of the first examples one considers when studying the box and Assouad dimensions (in particular $F_1$) and elementary calculations reveal that for any $\lambda>0$
\[
\dim_\mathrm{B} F_\lambda \ =  \ \frac{1}{\lambda+1} \ < \ 1 = \dim_\mathrm{A} F_\lambda.
\]
We therefore have the following immediate Corollary of Theorem \ref{sequencesmain}.

\begin{cor}
For all $\lambda>0$ and  $\theta \in (0,1)$ we have
$$\dim_\mathrm{A}^{\theta} F_\lambda \ = \  \frac{1}{(\lambda+1)(1-\theta)}\wedge 1.$$
\end{cor}

\begin{figure}[H] 
	\centering
	\includegraphics[width=146mm]{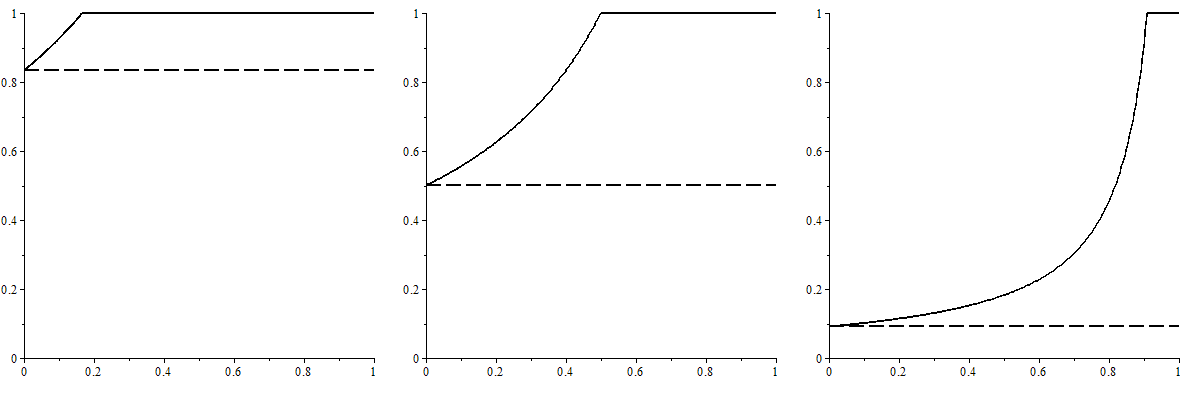}
\caption{Three plots of the Assouad spectrum of $F_\lambda$ for $\lambda = 1/5, 1, 10$ going from left to right.  The bounds from Proposition \ref{BB} are shown as dashed lines, although the upper bound is obtained in each case. \label{sequences}}
\end{figure}

The family of sets $\{F_\lambda\}_{\lambda>0}$ studied in this section provide us with a simple continuum of sets with the property that any one can be mapped onto any other by a bi-H\"older map.  Since we have a very simple explicit formula for the Assouad spectrum of each $F_\lambda$, this provides an excellent opportunity to test the bounds obtained in Corollary \ref{Holdercor}.  For $\alpha>0$, let $S_\alpha : [0,1] \to [0,1]$ be defined by $S_\alpha(x) = x^\alpha$ and observe that, for any $\lambda>0$, we have $S_\alpha(F_\lambda) \ = \ F_{\alpha \lambda}$.  Also note that if $\alpha \in (0,1)$ then $S$ is $\alpha$-H\"older with a Lipschitz inverse,  and if $\alpha >1$ then $S$ is Lipschitz  with a $1/\alpha$-H\"older inverse. In the $\alpha \in (0,1)$ region, the bounds on the Assouad spectrum from Corollary \ref{Holdercor} yield that for $\theta \in (0,1)$
\[
\frac{1-\theta/\alpha}{(1-\theta)}\dim_\mathrm{A}^{\theta/\alpha}F_\lambda  \, \vee \,   \overline{\dim}_\mathrm{B} F_\lambda      \ \leq \  \dim_\mathrm{A}^{\theta} S_\alpha(F_\lambda) \  \leq \  \frac{1-\alpha\theta}{\alpha(1-\theta)}\dim_\mathrm{A}^{\alpha\theta} F_\lambda  \, \wedge \, 1
\]
and applying the explicit formulae for the Assouad spectra derived above gives:
\begin{eqnarray*}
\frac{1-\theta/\alpha}{(1-\theta)} \left( \frac{1}{(\lambda+1)(1-\theta/\alpha)} \wedge 1 \right)  \, \vee \,   \frac{1}{(\lambda+1)}   &\leq&  \frac{1}{(\alpha\lambda+1)(1-\theta)} \wedge 1  \\ \\
&\leq& \frac{1-\alpha\theta}{\alpha(1-\theta)}   \, \wedge \,   \frac{1}{\alpha (\lambda+1)(1- \theta)}   \wedge 1 
\end{eqnarray*}
\begin{figure}[H] 
	\centering
	\includegraphics[width=120mm]{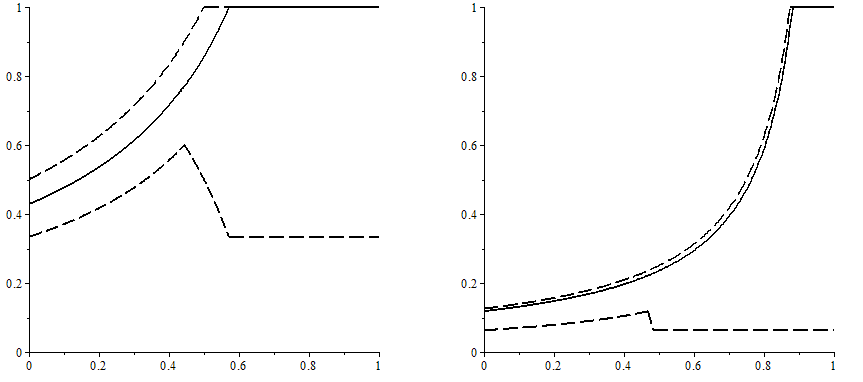}
\caption{Two plots of the bounds on the Assouad spectrum of the polynomial sequence $F_\lambda$ under bi-H\"older distortion by $S_\alpha$.  On the left $\lambda=2$ and $\alpha=2/3$ and on the right $\lambda=15$ and $\alpha=1/2$.  The actual spectrum of $S_\alpha(F_\lambda)$ is shown as a solid line and the bounds are dashed.\label{Holderpic1}}
\end{figure}
In the $\alpha >1$ region, the bounds on the Assouad spectrum from Corollary \ref{Holdercor} yield that for $\theta \in (0,1)$
\[
\frac{1-\alpha\theta}{\alpha(1-\theta)}\dim_\mathrm{A}^{\alpha\theta}F_\lambda  \, \vee \,   \frac{\overline{\dim}_\mathrm{B} F_\lambda}{\alpha}      \ \leq \  \dim_\mathrm{A}^{\theta} S_\alpha(F_\lambda) \  \leq \  \frac{1-\theta/\alpha}{1-\theta}\dim_\mathrm{A}^{\theta/\alpha} F_\lambda  \, \wedge \, 1
\]
and applying the explicit formulae for the Assouad spectra derived above gives:
\begin{eqnarray*}
\frac{1-\alpha\theta}{\alpha(1-\theta)} \left( \frac{1}{(\lambda+1)(1-\alpha\theta)} \wedge 1 \right)  \, \vee \,   \frac{1}{\alpha (\lambda+1)}   &\leq&  \frac{1}{(\alpha\lambda+1)(1-\theta)} \wedge 1  \\ \\
&\leq& \frac{1-\theta/\alpha}{1-\theta}   \, \wedge \,   \frac{1}{(\lambda+1)(1- \theta)}   \wedge 1 
\end{eqnarray*}

\begin{figure}[H] \label{Holderpic2}
	\centering
	\includegraphics[width=120mm]{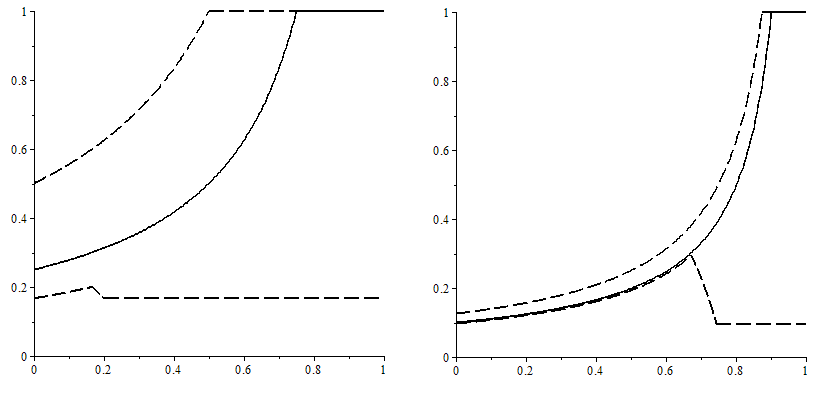}
\caption{Two plots of the bounds on the Assouad spectrum of the polynomial sequence $F_\lambda$ under bi-H\"older distortion by $S_\alpha$.  On the left $\lambda=1$ and $\alpha=3$ and on the right $\lambda=7$ and $\alpha=13/10$. The actual spectrum of $S_\alpha(F_\lambda)$ is shown as a solid line and the bounds are dashed.}
\end{figure}

\section{Unwinding spirals}  \label{SpiralsSect}

In this section we consider the problem of `unwinding spirals' or, more precisely, the question of whether a given spiral can be mapped to a unit line segment via a homeomorphism with certain `metric' restrictions; see \cite{spirals} for an overview of recent results in this direction.  A classical positive result is that the `logarithmic spiral' can be `unwound' to a unit line segment by a bi-Lipschitz map, see \cite{unwindspirals}.  Moreover, this is sharp in the sense that if the winding rate is slower than logarithmic, then a simple length estimate shows that the spiral cannot be unwound by a bi-Lipschitz map.  In particular, the spiral has infinite length and bi-Lipschitz images of sets with infinite length must have infinite length.

We are interested in the more general question of whether a spiral can  be unwound via a bi-H\"older homeomorphism and what restrictions are there on the bi-H\"older parameters?  Our main result is in the negative direction: we show that if the bi-H\"older map is too close to being bi-Lipschitz, then it cannot unwind the spiral, where `too close' is precisely characterised by the upper box dimension of the spiral.  Our result is a simple application of our work on how the Assouad spectrum can change under bi-H\"older maps and, moreover, we show that one gets strictly better information than if one considers the box or Hausdorff dimensions directly.

In general a \emph{spiral} $S$ is defined to be the set:
\[
S = \{ \phi(\alpha)\exp(i \alpha) \, : \, \alpha \in [0, \infty) \} \, \cup \{0 \}
\]
where $\phi$ is any continuous decreasing real-valued function such that $\lim_{\alpha\rightarrow \infty}\phi(\alpha)=0$ and for convenience we assume that $\phi(0) = 1$.  We say a spiral is \emph{convex differentiable} if $\phi$ is differentiable and its derivative is non-decreasing.  We also say that the spiral has  \emph{monotonic winding} if the function
\[
x \mapsto \phi(x) - \phi(x+2\pi)
\]
is decreasing in $x \geq 0$.  This is similar to the assumption that decreasing sequences have decreasing gaps.  Indeed, monotonic winding  guarantees that any ray starting at the origin intersects the spiral in a decreasing sequence with decreasing gaps.  If $\phi(x) = \exp(-cx)$ for some $c>0$, then the resulting spiral is often referred to as the \emph{logarithmic spiral} (mentioned above).  If
\[
\frac{\log\phi(x)}{x} \to 0
\]
as $x \to \infty$, then the winding is said to be \emph{sub-exponential}.  This is the interesting case since then the spiral has infinite length and thus cannot be unwound by a bi-Lipschitz homeomorphism.
\begin{figure}[H] 
	\centering
	\includegraphics[width=148mm]{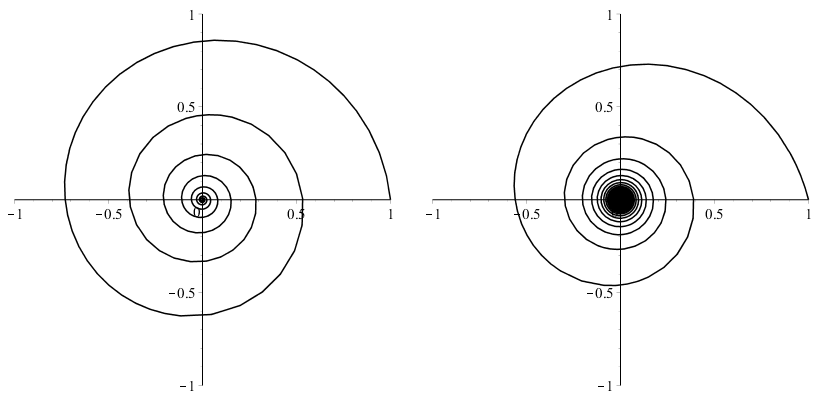}
\caption{Two spirals: on the left, the \emph{logarithmic spiral} with $c=1/10$; and, on the right, a spiral with sub-exponential winding where $\phi(x) = 1/(x/4+1)$.\label{Holderpic2}}
\end{figure}
First we prove a dichotomy for spirals with monotonic winding. 
\begin{thm} \label{assspiral}
Let $S$ be a spiral with monotonic winding.  If the winding is sub-exponential, then $\dim_\mathrm{A} S = 2$ and otherwise $\dim_\mathrm{A} S = 1$.
\end{thm}
To prove this result we prove that if the Assouad dimension is strictly  less than 2, then the winding must be exponential.  Note that this is yet another proof of the fact that spirals with sub-exponential winding cannot be unwound to a line segment by a bi-Lipschitz homeomorphism, since such maps preserve Assouad dimension.

\begin{thm} \label{spiral}
Let $S$ be a convex differentiable spiral with monotonic winding  such that $\overline{\dim}_\mathrm{B} S>1$.  Then for all $\theta \in (0,1)$ we have
\[
\dim_{\mathrm{A}}^\theta S \, = \, \frac{\overline{\dim}_\mathrm{B} S}{1- \theta} \, \wedge \, 2.
\]
Note that such spirals must have sub-exponential winding.
\end{thm}

Clearly any spiral $S$ is homeomorphic to the unit line segment $[0,1]$.  Let $f$ be a homeomorphism between these two sets and suppose that $f$ is also bi-H\"older  with parameters $\beta\geq 1\geq\alpha>0$, i.e. for all $x,y \in S$ we have
$$|x-y|^\beta \ \lesssim \  |f(x)-f(y)| \ \lesssim \  |x-y|^\alpha.$$
It follows immediately from the standard results for box dimension, that
\[
\frac{1}{\beta} \, \overline{\dim}_\mathrm{B}S    \    \leq \  \overline{\dim}_\mathrm{B} f(S) \, = \, \overline{\dim}_\mathrm{B} [0,1] \, = \, 1 
\]
and so, provided the spiral has upper box dimension larger than its topological dimension,  $f$ must be quantitatively far away from being bi-Lipschitz.  In particular, we require
\[
\beta \,  \geq \,  \overline{\dim}_\mathrm{B}S \, > \, 1.
\]
Observe that if we consider Hausdorff dimension here, then  we get no information on $\beta$ because the Hausdorff dimension of \emph{any} spiral is 1.  Fortunately, we can get more information if we consider the Assouad dimension.

\begin{cor} \label{spiralcor}
Let $f$ be a bi-H\"older homeomorphism with parameters $\beta\geq 1\geq\alpha>0$ mapping a convex differentiable spiral $S$ with sub-exponential and monotonic winding to a line segment.  Also, assume that $\overline{\dim}_\mathrm{B}S  >  1$.   Then
\[
\beta  \ \geq \   \alpha \, + \,  \overline{\dim}_\mathrm{B}S  \, \left(1-\frac{\alpha}{2}\right)  \  \geq \ (1+\alpha/2) \vee \overline{\dim}_\mathrm{B}S  \, >  \, 1
\]
and,  if $\alpha = 1$, then
\[
\beta \,  \geq \,  1+ \frac{\overline{\dim}_\mathrm{B}S}{2}.
\]
\end{cor}

\begin{proof}
It follows from Theorem \ref{Holderassouadcor} that
\[
\dim_\mathrm{A} f(S) \  \geq \  \frac{\dim_\mathrm{A} S}{\beta-\theta_0\alpha}(1-\theta_0)
\]
where
\[
\theta_0=\inf\left\{ \theta\in [0,1] : \dim_\mathrm{A}^{\theta} S=\dim_\mathrm{A} S \right\} = 1 - \frac{\overline{\dim}_\mathrm{B}S}{2}
\]
by  Theorem \ref{spiral}.  Therefore, by Theorem \ref{spiral}, we have
\[
 1 \ \geq  \  \frac{2}{\beta-\alpha\left(1 - \frac{\overline{\dim}_\mathrm{B}S}{2} \right)} \frac{\overline{\dim}_\mathrm{B}S}{2} 
\]
and solving for $\beta$ yields
\[
\beta  \ \geq \   \alpha \, + \,  \overline{\dim}_\mathrm{B}S  \, \left(1-\frac{\alpha}{2}\right) 
\]
as required.
\end{proof}

Note that Corollary \ref{spiralcor} gives strictly better information than we get directly from the upper box dimension, provided the upper box dimension of $S$ is strictly less than 2.  Otherwise, both estimates reduce to $\beta \geq 2$.

For the purpose of the following proofs, define $f(x)=\phi(2\pi x)$ and $g(x)=f(x)-f(x+1)$.  By definition $f$ is decreasing and monotonic winding guarantees that $g$ is also decreasing.  Clearly  $g(x)<f(x)$ for any $x\geq 0$.  We will refer to the $\delta$-neighbourhood of $S$ (or part of $S$)  as a $\delta$-\emph{sausage} for $\delta>0$.

\subsection{Proof of Theorem \ref{assspiral}}

  Let $0<r<R<g(0)<1$   and observe that there is a unique $x_r>0$ such that $g(x_r)=r$, and a  unique $x_R>0$ such that $f(x_R)=R$.  If $x_r<x_R$, we see that the $r$-sausage of $S$ will  completely cover the ball $B(0,R)$, thus the number of $r$-balls needed to cover $B(0,R)\cap S$ is $ \gtrsim (R/r)^2$.    If $x_r>x_R$, the $r$-sausage of $S$ will completely cover a ball smaller than $B(0,R)$. It is easy to see that the smaller ball can be taken to be $B(0,f(x_r))$, thus we need $\gtrsim (f(x_r)/r)^2$ many $r$-balls to cover $B(0,R)\cap S$.

Since $g$ and $f$ are continuous decreasing functions and $g<f$ we can deduce that the inverse functions $g^{-1},f^{-1}$ are also decreasing and $g^{-1}<f^{-1}$ (on the appropriate domain).  Therefore $g^{-1}(R)<f^{-1}(R)$ and there exists a unique $r\in (0,R)$ such that $g^{-1}(r)=f^{-1}(R)$. We see that the number of  $r$-balls needed to cover $B(0,R)\cap S$ is 
$$ \gtrsim \left( f(g^{-1}(r))/r\right)^2=(R/r)^2$$
Therefore, if the Assouad dimension of $S$ is strictly smaller than $2$, then it must be true that for all small enough $R>0$, the $r$ defined above must be such that $R/r$ is uniformly bounded from above. This reasoning is similar to the case of decreasing sequences.  Suppose there exists $M>1$ such that  $R/r<M$ for all small enough $R$ and the $r$ chosen above.  Then we conclude that
$$g^{-1}(R)<g^{-1}(r)=f^{-1}(R)<g^{-1}(R/M)$$
and so, applying $g$ throughout, we get 
$$R>g(f^{-1}(R))>R/M.$$
Observe  that $f^{-1}(R)=x_R$ and so
\[
f(x_R) > g(x_R) > f(x_R)/M
\]
and subtracting $f(x_R)$ throughout and taking negatives yields
$$f(x_R+1)<f(x_R)(1-1/M).$$
In particular this holds for sufficiently small $R>0$ and so by continuity of $f$ we can deduce that
$$f(x+1)<f(x)(1-1/M)$$ for any large enough $x$.
From here it is clear that
\[
\lim_{x\rightarrow \infty}|\log f(x)/x|>0
\]
because the limit holds for integral $x$, and since $f(x)$ is a decreasing function the limit holds in general. Therefore $f$, and hence $\phi$, is at least exponential and since spirals with exponential winding are bi-Lipschitz equivalent to a line segment we deduce that $\dim_\mathrm{A} S = 1$.

\subsection{Proof of Theorem \ref{spiral}}

 Denote the upper box dimension of $S$ by $B$ and recall that by assumption and Theorem \ref{assspiral} we know $1< B\leq 2 = \dim_\mathrm{A} S$.  We will prove that
\[
\dim_{\mathrm{A}}^{\theta} S\geq \frac{B}{1-\theta}
\]
for $0< \theta < 1-B/2$ which, combined with Proposition \ref{BB} and Corollary \ref{stayA}, proves the result.   Fix such a $\theta$. If we can find a sequence of $r_i\to 0$ such that
\[
g^{-1}(r_i)\leq f^{-1}(r_i^{1-B^-/2})
\]
then we would have:
$$r_i^{-B^-}\lesssim \left( f(g^{-1}(r_i))/(r_i)\right)^2\lesssim N(B(0,r_i^\theta)\cap S,r_i)$$
which proves the result.  Therefore assume that for all $r>0$ small enough we have the inequality
$$g^{-1}(r)>f^{-1}(r^{1-B^-/2})$$
which also implies that for all $r>0$ small enough
$$g^{-1}(r)>f^{-1}(r^{\theta})$$
since $\theta < 1-B/2$ and $f^{-1}$ is decreasing.  This also implies  that $f(x)<g(x)^{1-B^-/2}$ for all large enough $x$.  We assumed that $f(x)$ is differentiable and convex and therefore for all large enough $x>0$ we have by the mean value theorem that
$$f(x)^{\frac{1}{1-B^-/2}}<g(x)=f(x)-f(x+1)=-f'(\zeta)<-f'(x)$$
where $\zeta \in [x,x+1]$.  In fact it is true that for a suitable constant $C>0$ and all $x>0$ we have
$$Cf(x)^{\frac{1}{1-B^-/2}}<-f'(x).$$
For simplicity we write $\alpha = \frac{1}{1-B^-/2}>1$.  We have
$$-C>\frac{f'(x)}{f(x)^{\alpha}}=\frac{1}{1-\alpha}\left(f(x)^{1-\alpha}\right)'$$
which gives
$$\left(f(x)^{1-\alpha}\right)'>C(\alpha-1).$$
Integrating both sides of this inequality yields
$$f(x)^{1-\alpha}-f(0)^{1-\alpha}>C(\alpha-1)x$$
and this implies for all large enough $x$ that
\begin{equation} \label{nicespiralbound}
f(x)<\left(\frac{1}{C(\alpha-1)x+f(0)^{1-\alpha}}\right)^{\frac{1}{\alpha-1}}\lesssim x^{-\frac{1}{\alpha-1}}=x^{-\frac{1-B^-/2}{B^-/2}}
\end{equation}
and therefore
\begin{equation} \label{nicespiralbound-1}
f^{-1}(x)\lesssim x^{-\frac{B^-/2}{1-B^-/2}}.
\end{equation}
We now move towards bounding $N(S\cap B(x,r^\theta),r)$. First of all we need
\[
\gtrsim\left( f(g^{-1}(r))/(r)\right)^2
\]
 many $r$-balls to cover
\[
\{ \phi(\alpha)\exp(i \alpha) \, : \, \alpha \in [2\pi g^{-1}(r),\infty) \}
\]
because the $r$-sausage will cover $B(0,g^{-1}(r))$ completely.  Now consider the subset of $S$ given by
\[
\{ \phi(\alpha)\exp(i \alpha) \, : \, \alpha \in \left(2\pi f^{-1}(r^\theta), 2\pi g^{-1}(r)\right)\} ,
\]
and decompose part of this set into the disjoint union of the sets
\[
S_m=\{ \phi(\alpha)\exp(i \alpha) \, : \, \alpha \in (2\pi m,2\pi (m+1) ) \},
\]
 over integers $[f^{-1}(r^\theta)]+1\leq m\leq [g^{-1}(r)]$, where $[x]$ denotes the integer part of $x >0$.   The projection of each $S_m$ onto the real axis contains  an interval of length $f(m)$, so to cover $S_m$ we need at least $ \gtrsim f(m)/r$ many $r$-balls. Since distinct sets $S_m$ are at least $r$ separated, in order to cover all the sets in the above union we need at least
$$ \gtrsim \  \frac{1}{r}\sum_{k=1}^{[g^{-1}(r)]-[f^{-1}(r^\theta)]}f([f^{-1}(r^\theta)]+k)$$
many $r$-balls.   The sequence $f([f^{-1}(r^\theta)]+k)$ for $k=1, \dots, [g^{-1}(r)]-[f^{-1}(r^\theta)]$ is decreasing and  the minimum of the sequence  is $f([g^{-1}(r)])\geq f(g^{-1}(r)) \geq r$. Since $g$ is also decreasing we see that
\[
f([f^{-1}(r^\theta)]+k)-f([f^{-1}(r^\theta)]+k+1) \ \geq \  r
\]
 for $k=1, \dots, [g^{-1}(r)]-[f^{-1}(r^\theta)]-1$.  Therefore we have
\[
f([f^{-1}(r^\theta)]+[g^{-1}(r)]-[f^{-1}(r^\theta)]-k) \ \geq \   (k+1) \, r
\]
and applying this inequality to the above sum yields
\begin{eqnarray*}
\sum_{k=1}^{[g^{-1}(r)]-[f^{-1}(r^\theta)]}f([f^{-1}(r^\theta)]+k)  &\geq & \sum_{k=1}^{[g^{-1}(r)]-[f^{-1}(r^\theta)]} k  r \\ \\
&=& \frac{1}{2}([g^{-1}(r)]-[f^{-1}(r^\theta)])([g^{-1}(r)]-[f^{-1}(r^\theta)]+1) \, r\\ \\
&\gtrsim & (g^{-1}(r)-f^{-1}(r^\theta))^2 r
\end{eqnarray*}
where, in particular, the last $\gtrsim$ holds for all $r\to 0$. Therefore we have
\begin{equation} \label{funnybound} 
N(B(0,r^\theta)\cap S,r)\gtrsim (g^{-1}(r)-f^{-1}(r^\theta))^2.
\end{equation}
Now we will take into consideration the box dimension of the spiral.  We have
\[
N(S,r)\lesssim\left( f(g^{-1}(r))/(r)\right)^2+L/r
\]
where $L$ is the length of the rectifiable part of the spiral corresponding to angles $0$ to $2\pi g^{-1}(r)$.  We may bound $L$ from above by employing the classical length formula:
\begin{eqnarray*}
L &=& \int_0^{2\pi g^{-1}(r)}\sqrt{\phi^2+\dot{\phi}^2} \ d\alpha \\ \\
&=& \int_{K_1}\sqrt{\phi^2+\dot{\phi}^2} \ d\alpha \ + \  \int_{K_2} \sqrt{\phi^2+\dot{\phi}^2} \ d\alpha
\end{eqnarray*}
where $K_1=\{ \alpha \in (0, 2\pi g^{-1}(r)) \ : \  |\phi|>|\dot{\phi}|\}$, $K_2=\{\alpha \in (0, 2\pi g^{-1}(r)) \ : \  |\phi|\leq |\dot{\phi}|\}$. By splitting the integral in this way we obtain
\begin{eqnarray*}
L&\leq& \sqrt{2}\int_{K_1}\phi \, d\alpha \ - \ \sqrt{2}\int_{K_2}\dot{\phi} \,  d\alpha \\ \\
&\leq&   \sqrt{2}\int_0^{2\pi g^{-1}(r)}\phi \,  d\alpha \ + \ \sqrt{2} \,  \phi(0)\\ \\
&\leq&   \sqrt{2}\sum_{k=0}^{[g^{-1}(r)+1]} f(k) \ + \ \sqrt{2}
\end{eqnarray*}
where the last inequality comes from the fact that $f$ is decreasing. Therefore
$$N(S,r) \ \lesssim \ \left( f(g^{-1}(r))/(r)\right)^2 \ + \  \sum_{k=0}^{[g^{-1}(r)+1]} f(k)/r\ + \ 1/r.$$
Using (\ref{nicespiralbound}) we can bound the middle term above by
$$\sum_{k=0}^{[g^{-1}(r)+1]} f(k)/r \  \lesssim \  \sum_{k=0}^{[g^{-1}(r)+1]} k^{-\frac{1-B^-/2}{B^-/2}}/r \  \lesssim  \  r^{-1}g^{-1}(r)^{1-\frac{1-B^-/2}{B^-/2}}.$$ 
Therefore 
$$N(S,r) \ \lesssim \  \left( f(g^{-1}(r))/(r)\right)^2+r^{-1}g^{-1}(r)^{1-\frac{1-B^-/2}{B^-/2}}+1/r.$$
Since the (upper) box dimension of $S$ is $B$ we can find a sequence of $r_i\to 0$ such that
$$\left( f(g^{-1}(r_i))/(r_i)\right)^2+r_i^{-1}g^{-1}(r_i)^{1-\frac{1-B^-/2}{B^-/2}}+1/r_i \ \gtrsim \  N(S,r_i) \ \gtrsim \  r_i^{-B^-}.$$ 
Since $B>1$ we can assume $B^->1$, and then either the first term or the second term is $\gtrsim r_i^{-B^-}$ for infinitely many $i$.  If this is true for the first term then by our initial observation we have $\dim_{\mathrm{A}}^{\theta} S\geq \frac{B^-}{1-\theta}$.  If this is true for the second term, then we deduce that for infinitely many $i$ we have
$$g^{-1}(r_i)\gtrsim r_i^{-B^-/2}.$$
Recalling the lower bound (\ref{funnybound}) and the assumption that $\theta<1-B^-/2$, we conclude that
\begin{eqnarray*}
N(B(0,r_i^\theta)\cap S,r_i) &\gtrsim & (g^{-1}(r_i)-f^{-1}(r^\theta_i))^2 \\ \\
 &\gtrsim & \left(r_i^{-B^-/2}-r_i^{-\theta\frac{B^-/2}{1-B^-/2}}\right)^2  \qquad \text{by (\ref{nicespiralbound-1})}\\ \\
 &\gtrsim & r_i^{-B^-}
\end{eqnarray*}
which completes the proof.

\section{An example with non-monotonic spectra} \label{NON-MONO2}

 In this section we construct sets whose spectra exhibit strange properties.  In particular, we prove that the spectra are not necessarily monotone; monotonicity can be broken infinitely many times; and, both the Assouad and lower spectra can have infinitely many phase transitions, i.e. points where they fail to be differentiable.

Given an interval $I$ of length $L<1$ and numbers $\alpha>\beta>1$, we construct a set $F_{\alpha,\beta}\subset I$ via the following inductive procedure:
\begin{itemize}
\item[1st step.] We pack closed intervals of length $L^{\alpha}$ with gaps of length $L^{\beta}$ inside $L$.  The particular way of packing does not matter as long as it is by an optimal number.  We let the union of the closed intervals of length $L^{\alpha}$ be denoted by $I_1$.
\item[2nd step.] For each interval of length $L^{\alpha}$ appearing at the first step, we optimally  pack  intervals of length $L^{\alpha^2}$ with gaps of length $L^{\alpha\beta}$.  We let the union of the closed intervals of length $L^{\alpha^2}$ be denoted by $I_2$.
\item[$k$th step.] For each interval of length $L^{\alpha^{k-1}}$ appearing at the $(k-1)$th step, we optimally  pack  intervals of length  $L^{\alpha^k}$with gaps of length $L^{\alpha^{k-1}\beta}$.  We let the union of the closed intervals of length $L^{\alpha^k}$ be denoted by $I_k$.
\end{itemize}
We obtain a nested sequence of compact sets  $I\supset I_1\supset I_2\supset I_3 \supset \dots$ and finally we let
\[
F_{\alpha,\beta}=\bigcap_{k=1}^\infty  I_k
\]
which is a non-empty compact set.  We can compute the dimensions of $F_{\alpha,\beta}$, but we leave the details to the reader.  In particular, 
$$\dim_{\mathrm{L}} F_{\alpha,\beta} \ = \ 0 \ < \  \underline{\dim}_\mathrm{B} F_{\alpha,\beta} \ = \ \frac{\beta-1}{\alpha-1}  \ < \   \overline{\dim}_\mathrm{B} F_{\alpha,\beta} \ = \ \frac{\beta-1}{\alpha-1}\frac{\alpha}{\beta} \  < \   \dim_{\mathrm{A}} F_{\alpha,\beta} \ = \ 1.$$
The exact computation of the spectrum is complicated, but nevertheless we can get some information  without much effort.

\begin{lma} \label{L1}
Whenever $\log_{\alpha} \frac{1}{\theta}$ is an integer, we have
$$\dim_{\mathrm{L}}^{\theta} F_{\alpha,\beta}=\frac{\beta-1}{\alpha-1} =  \underline{\dim}_\mathrm{B} F_{\alpha,\beta}$$
and
$$\dim_{\mathrm{A}}^{\theta} F_{\alpha,\beta}=\frac{\beta-1}{\alpha-1}\frac{\alpha}{\beta} =  \overline{\dim}_\mathrm{B} F_{\alpha,\beta}.$$
In particular, the Assouad/lower spectrum is equal to the  upper/lower box dimension infinitely often.  This is the general lower/upper bound from Proposition \ref{BB}/\ref{BBL}.
\end{lma}

\begin{proof}
Suppose $\log_{\alpha} \frac{1}{\theta}=m$ is an integer.  For any $R\in (0,L)$, there is a unique integer $k \geq 1$ such that $$L^{\alpha^k}\leq R< L^{\alpha^{k-1}}.$$  
In particular, this means that
$$L^{\alpha^{k+m}}\leq R^{1/\theta}< L^{\alpha^{k-1+m}}.$$
There are two different cases to consider:
\begin{enumerate}
\item {$L^{\alpha^k}\leq R< L^{\alpha^{k-1}\beta}$ and $L^{\alpha^{k+m}}\leq R^{1/\theta}< L^{\alpha^{k-1+m}\beta}$}.

For any $x\in F_{\alpha,\beta}$, the ball $B(x,R)$  contains one $L^{\alpha^k}$ interval, and any $R^{1/\theta}$-ball can cover at most one ${L^{\alpha^{k+m}}}$ interval because the distance between  two disjoint intervals in the construction is $L^{\alpha^{k-1+m}\beta}$.  By construction, in each $L^{\alpha^k}$ interval the number of intervals of length $L^{\alpha^{k+m}}$ is
$$\asymp \ [L^{\alpha^k(1-\beta)}][L^{\alpha^{k+1}(1-\beta)}][L^{\alpha^{k+2}(1-\beta)}]\dots[L^{\alpha^{k+m}(1-\beta)}]$$
where the $[.]$ denotes the integer part. Therefore, the above argument shows that
$$N(B(x,R),R^{1/\theta})\lesssim L^{(\alpha^k+\alpha^{k+1}+\alpha^{k+2}+\dots+\alpha^{k+m-1})(1-\beta)}=L^{\alpha^k\frac{1-\beta}{1-\alpha}(1-\alpha^m)} .$$
We now derive a  lower bound for $N(B(x,R),R^{1/\theta})$.  Recall we only need the behaviour of
\[
N(B(x,R),R^{1/\theta})
\]
when $R$ is sufficiently small.  Choose $R$ so small such that $k$  will satisfy
$$[L^{\alpha^{k}(1-\beta)}]>cL^{\alpha^{k}(1-\beta)}$$
where $c>\frac{1}{2^{1/m}}$.  This is possible because $[x]\geq x(1-1/x)$ for any positive $x$, and for $x$  large enough we have $1-1/x>\frac{1}{2^{1/m}}$.  Then for all small enough $R>0$ we have $$N(B(x,R),R^{1/\theta})\gtrsim c^m L^{\alpha^k\frac{1-\beta}{1-\alpha}(1-\alpha^m)}\geq \frac{1}{2}L^{\alpha^k\frac{1-\beta}{1-\alpha}(1-\alpha^m)}.$$
In summary, for all sufficiently small $R$ satisfying the conditions of case (1) we have
$$N(B(x,R),R^{1/\theta}) \ \asymp \  L^{\alpha^k\frac{1-\beta}{1-\alpha}(1-\alpha^m)}.$$

\item {$L^{\alpha^{k-1}\beta}\leq R< L^{\alpha^{k-1}}$ and $L^{\alpha^{k+m-1}\beta}\leq R^{1/\theta}< L^{\alpha^{k-1+m}}$.}

In this case any ball $B(x,R)$   contains $$\left[\frac{R}{L^{\alpha^{k-1}\beta}}\right]\pm 1$$ many intervals of length $L^{\alpha^k}$.  Also to cover any interval of length $L^{\alpha^{k+m-1}}$ we need 
$$\left[\frac{L^{\alpha^{k+m-1}}}{R^{1/\theta}}\right]\pm 1$$ many $R^{1/\theta}$-balls.  Using the same tricks as in case (1), we may get rid of the integer part and the $\pm 1$. In summary, for all sufficiently small $R$ satisfying the conditions of case (2) we obtain
\begin{eqnarray*}
N(B(x,R),R^{1/\theta}) &\asymp & \frac{R}{L^{\alpha^{k-1}\beta}}L^{(\alpha^k+\alpha^{k+1}+\alpha^{k+2}+\dots+\alpha^{k+m-2})}\frac{L^{\alpha^{k+m-1}}}{R^{1/\theta}} \\ \\
&=& \frac{R}{R^{1/\theta}}L^{\alpha^k\frac{1-\beta}{1-\alpha}(1-\alpha^{m-1})+\alpha^{k+m-1}-\alpha^{k}\beta}.
\end{eqnarray*}
\end{enumerate}
Finally, applying the definition of $k$ yields the desired result.
\end{proof}

\begin{lma} \label{L2}
For every $\theta'$ such that $\log_{\alpha} \frac{1}{\theta}=m$ is an integer, there exists $\epsilon=\epsilon(m)>0$ such that
$\dim_{\mathrm{L}}^{\theta} F_{\alpha,\beta}$ and $\dim_{\mathrm{A}}^{\theta} F_{\alpha,\beta}$ are not constant in the interval $[\theta'(1-\varepsilon),\theta']$.
\end{lma}

\begin{proof}
Again, for any $R>0$, there is an integer $k$ such that
$$L^{\alpha^k}\leq R< L^{\alpha^{k-1}}.$$
The calculation may now proceed as in the proof of Lemma \ref{L1} but with $\theta$ very close to, but smaller than, $\theta'$.  Write $\log_{\alpha}\frac{1}{\theta}=m+c$ where $m$ is the integer part and $c$ is the fractional part which is assumed to be small.  We have now four cases, which we consider in turn.   In each case, we drop the integer part symbols $[.]$ because we are only interested in the asymptotic behaviour of $N(B(x,R),R^{1/\theta})$.

\begin{enumerate}
\item $L^{\alpha^k}\leq R< L^{\alpha^{k-c}}$ and  $L^{\alpha^{k+m}\beta}\leq R^{1/\theta}< L^{\alpha^{k+m}}.$
Any ball $B(x,R)$ with $x\in F_{\alpha,\beta}$ will contain one interval of length  $L^{\alpha^k}$ and, for any $L^{\alpha^{k+m}}$ interval, we need approximately $$\frac{L^{\alpha^{k+m}}}{R^{1/\theta}}$$ many $R^{1/\theta}$-balls to cover it. Therefore
$$N(B(x,R),R^{1/\theta}) \ \asymp \  \frac{L^{\alpha^{k+m}}}{R^{1/\theta}}L^{\alpha^k\frac{1-\beta}{1-\alpha}(1-\alpha^m)}.$$
\item $L^{\alpha^{k-c}}\leq R< L^{\alpha^{k-1}\beta}$ and  $L^{\alpha^{k+m}}\leq R^{1/\theta}< L^{\alpha^{k+m-1}\beta}.$
Any $B(x,R)$ with $x\in F_{\alpha,\beta}$ will contain one  $L^{\alpha^k}$ interval and, on the other hand, every $L^{\alpha^{k+m}}$ interval contained inside this $L^{\alpha^k}$ interval needs one $R^{1/\theta}$-ball to cover it. Therefore
$$N(B(x,R),R^{1/\theta})=L^{\alpha^k\frac{1-\beta}{1-\alpha}(1-\alpha^m)}.$$
\item $L^{\alpha^{k-1}\beta}\leq R< L^{\alpha^{k-1-c}\beta}$ and $L^{\alpha^{k+m}}\leq R^{1/\theta}< L^{\alpha^{k+m-1}\beta}.$
Any $B(x,R)$ with $x\in F_{\alpha,\beta}$ contains approximately $$\frac{R}{L^{\alpha^{k-1}\beta}}$$ many intervals of length $L^{\alpha^k}$, but for each interval of length $L^{\alpha^{k+m}}$ we need one $R^{1/\theta}$-ball to cover it. Therefore
$$N(B(x,R),R^{1/\theta}) \ \asymp \  \frac{R}{L^{\alpha^{k-1}\beta}}L^{\alpha^k\frac{1-\beta}{1-\alpha}(1-\alpha^m)}.$$
\item $L^{\alpha^{k-1-c}\beta}\leq R< L^{\alpha^{k-1}}$ and $L^{\alpha^{k+m-1}\beta}\leq R^{1/\theta}< L^{\alpha^{k+m-1}}.$
Any $B(x,R)$ with $x\in F_{\alpha,\beta}$ contains approximately $$\frac{R}{L^{\alpha^{k-1}\beta}}$$ many intervals of length $L^{\alpha^k}$.  Also for any $L^{\alpha^{k+m-1}}$ interval we need approximately $$\frac{L^{\alpha^{k+m-1}}}{R^{1/\theta}}$$ many $R^{1/\theta}$-balls to cover it.  Therefore
$$N(B(x,R),R^{1/\theta}) \ \asymp \  \frac{R}{R^{1/\theta}}L^{\alpha^k\frac{1-\beta}{1-\alpha}(1-\alpha^{m-1})+\alpha^{k+m-1}-\alpha^{k}\beta}.$$
\end{enumerate}

It follows that for $c$ sufficiently small we have the following formula for spectra
$$\dim_{\mathrm{A}}^{\theta} F_{\alpha,\beta}=\frac{\frac{\alpha}{\beta}\frac{1-\beta}{1-\alpha}\left(\alpha^c-\frac{1}{\theta}\right)-\alpha^c+1}{1-\frac{1}{\theta}}$$
and
$$\dim_{\mathrm{L}}^{\theta} F_{\alpha,\beta}=\frac{\frac{1-\beta}{1-\alpha}\left(\alpha^c-\frac{1}{\theta}\right)}{1-\frac{1}{\theta}}$$ 
which are not constant.  The above formulae are obtained by considering the 4 cases above along with estimates derived from the definition of $k$.  This an expression for the asymptotic behaviour of $N(B(x,R),R^{1/\theta})$ which yields formulae for the spectra.   These formulae only hold for $c$ smaller than some constant $c_0 \in (0,1)$ (independent of $m$). This means they  are valid for
\[
\frac{1}{\alpha^{m+c_0}} \leq \theta \leq \frac{1}{\alpha^m}
\]
for all positive integers $m$.
\end{proof}

The following corollary follows from the results (and proofs) given in this section.  It shows that the sets we construct here exhibit some new phenomena.

\begin{cor}
For the sets $F_{\alpha,\beta}$ constructed in this section, the Assouad and lower spectra have the following properties:
\begin{enumerate}
\item they are not monotonic and, moreover, there are infinitely many disjoint intervals within which they fail to be monotonic.
\item they have infinitely many points of non-differentiability.
\end{enumerate}
\end{cor}

\begin{proof}
This follows immediately by combining  Lemmas  \ref{L1} and \ref{L2} with the fact that the spectra are continuous, see Propositions \ref{Con} and \ref{Con2}.
\end{proof}

\section{Open problems and further work} \label{OpenSection}

In this section we collect several open questions concerning the work we have presented in this paper.

Corollary \ref{Con} and Theorem \ref{Con2} show that all of the spectra we consider are continuous in $\theta$, but many of our examples exhibit phase transitions, preventing the spectra from being any more regular globally.  However, in all of our examples the spectra are piecewise analytic and we wonder if this is always the case, either with finitely or countably many phase transitions.

\begin{ques}
Is it true that the dimension spectra are piecewise analytic, or at least piecewise differentiable?
\end{ques}

In Corollary \ref{Con} and Theorem \ref{Con2} we proved that the Assouad and lower spectra are Lipschitz when restricted to any closed subinterval of $(0,1)$. However, we have not ruled out the possibility of examples where the spectra exhibit less regularity on the whole domain.

\begin{ques}
Is it possible for the Assouad and lower spectra to fail to be Lipschitz, or even H\"older, on the whole interval $(0,1)$?
\end{ques}

In Proposition \ref{basic2sep} we proved that the modified lower dimension is stable under finite unions, provided the sets are `properly separated'.  Note that this property does not hold for the lower dimension.  We were unable to determine if the `properly separated' condition can be dropped.
\begin{ques}
Is it true that for subsets $E, F$ of a common metric space, we always have
\[
\dim_{\mathrm{ML}} E \cup F \ \leq \ \dim_{\mathrm{ML}} E \vee \dim_{\mathrm{ML}}  F ?
\]
The opposite inequality is a trivial consequence of monotonicity.
\end{ques}

Theorems  \ref{meas1}-\ref{meas3} proved that the Assouad and lower spectra are of Baire class 2 and thus Borel measureable.  We have not been able to determine if the modified lower dimension or modified lower spectrum are Borel measureable.

\begin{ques}
Are the modified lower dimension and  modified lower spectrum Borel measureable and, if so, which Baire classes do they belong to (if any)?
\end{ques}

Recall that in our study of spirals with sub-exponential and monotonic winding we needed to make the additional assumption that the upper box dimension was strictly larger than 1.  At first this might seem like a strange assumption, but it is the analogue of assuming that a decreasing sequence has positive box dimension.  Indeed, if a set has box dimension 0, then Corollary \ref{BD0} tells us that the Assouad spectrum is constantly equal to 0, thus hiding any strange properties which may occur in that case.  There is no analogous result here and it remains an interesting problem to investigate what can happen when the box dimension of a spiral is 1.  We suspect that Theorem \ref{spiral} no longer holds and that other phenomena are possible.

\begin{ques}
What can one say in general about the dimension spectra of spirals with sub-exponential and monotonic winding in the case where the box dimension of the spiral is 1?
\end{ques}

Also on the topic of spirals, we proved that if a straight line segment is mapped to a spiral with sub-exponential and monotonic winding, then the H\"older exponent of that map must satisfy certain restrictions based on the upper box dimension of the spiral.  It would be interesting to investigate the sharpness of this result.  A first step in this investigation could be the following question.

\begin{ques}
Can a  spiral with sub-exponential and monotonic winding and with upper box dimension strictly larger than 1  be mapped to line segment by a bi-H\"older map?  If so, what are the sharp bounds on the H\"older exponents?
\end{ques}

Once one has a reasonable notion of metric dimension, one may wish to consider how this dimension behaves under canonical geometric operations, such as orthogonal projections or sections (intersections with hyperplanes).  This theory is very well-developed for the Hausdorff dimension, starting with the classical paper of Marstrand \cite{Marstrand}, see also \cite{MattilaProjections} for the higher dimensional analogue and the survey papers \cite{MattilaSurvey, FalconerFraserJin} for more details and up-to-date references.  Roughly speaking, the philosophy behind Marstrand's Theorem and later developments is that if $F \subseteq \mathbb{R}^d$ has `dimension' $s \in [0,d]$, then the `dimension' of the projection of $F$ onto hyperplanes of dimension $k<d$ should be almost surely constant, with respect to the natural measure on the Grassmanian manifold.  In the case of the Hausdorff dimension, the almost sure value is the largest possible, namely $s \wedge k$.  However, for box dimension the almost sure constant is more subtle and given by a \emph{dimension profile}, introduced by Falconer and Howroyd see \cite{FalconerHowroyd1, FalconerHowroyd2}.    Recently Fraser and Orponen \cite{FraserOrponen} proved that the Assouad dimension does not follow in the spirit of Marstrand's Theorem in that it can attain multiple values with positive probability (under projection).   It would be interesting to consider these results for the Assouad spectrum, since it can be viewed as an interpolation between the Assouad and upper box dimension.

\begin{ques}
For a given $\theta \in (0,1)$, is the Assouad spectrum of given set almost surely constant under projection onto hyperplanes? 
\end{ques}

Independent of the answer to the above question, it seems likely that a spectrum of dimension profiles would play a role in the study of how the Assouad spectrum behaves under projection.

The key theme of this paper has been what happens when one fixes the relationship between the two scales $R$ and $r$ used in the definition of the Assouad dimension.  Of course, there are many ways to fix this relationship.  Indeed, let $\phi : [0,1] \to [0,1]$ be a decreasing continuous function such that for all $\phi(x) \leq x$ for all $x \in [0,1]$.  Then one may define the $\phi$-\emph{Assouad dimension} to be the analogue where the relationship between the two scales is fixed by always choosing $r = \phi(R)$.  We have studied the continuously parameterised family of functions $\phi(x) = x^{1/\theta}$ and it turns out that this really is the `correct' family to consider in order to develop a rich theory.  Indeed it follows from our results that if
\[
\frac{\log x}{\log\phi(x)} \to 0 \qquad (x \to 0)
\]
then the $\phi$-Assouad dimension coincides with the upper box dimension for any totally bounded set.  Moreover, if
\[
\frac{\log x}{\log\phi(x)} \to 1  \qquad (x \to 0)
\]
then the $\phi$-Assouad dimension coincides with the Assouad dimension for any set where the Assouad dimension  is `witnessed' by the Assouad spectrum (i.e. the spectrum reaches the Assouad dimension for some $\theta \in (0,1)$).   Therefore one will (usually) only obtain a rich theory for functions $\phi$ which have an intermediate behaviour, which leads one directly to our functions $\phi(x) = x^{1/\theta}$.  However,  sets for which  the Assouad dimension  is \emph{not} `witnessed' by the Assouad spectrum fall through the net in some sense. We propose the following programme  to deal with such examples.  For functions $\phi$ defined above, let
\begin{eqnarray*}
\dim_\mathrm{A}^\phi F &=& \inf \bigg\{ \alpha \  : \   (\exists C>0) \, (\exists \rho>0) \, (\forall 0<r \leq \phi(R) \leq R\leq \rho)    \\ \\
&\,& \qquad \qquad  \qquad \qquad  \sup_{x \in F}  N \big( B(x,R) \cap F , r\big) \ \leq \ C \left(\frac{R}{r}\right)^\alpha \bigg\}.
\end{eqnarray*}
Notice that this is not quite the definition we alluded to above because we only require $r \leq \phi(R)$, and not $r = \phi(R)$.  However, this seems more natural for what follows.  One now asks the question: how difficult is it to witness the Assouad dimension?  More precisely, the problem is to classify for which functions $\phi$ we have $\dim_\mathrm{A}^\phi F = \dim_\mathrm{A} F $.

\bibliographystyle{amsalpha}
\bibliography{}

\printindex

\end{document}